\documentclass[bj,numbers]{imsart}

\usepackage[OT1]{fontenc}
\usepackage{amsthm,amsmath,amsfonts,amssymb}
\usepackage[colorlinks,citecolor=blue,urlcolor=blue]{hyperref} 

\usepackage{graphicx,calc}
\usepackage[latin1]{inputenc}
\usepackage{verbatim}
\usepackage{booktabs}
\usepackage{mathtools}
\usepackage{bbm}
\usepackage{url}
\usepackage{nicefrac}
\usepackage{multirow}

\graphicspath{{.}{figs/}}

\startlocaldefs
\theoremstyle{plain}
\newtheorem{Proposition}{Proposition}
\newtheorem{Lemma}{Lemma}
\newtheorem{Theorem}{Theorem}
\newtheorem{Corollary}{Corollary}

\theoremstyle{remark}             
\newtheorem{Remark}{Remark}

\newtheorem{example}{Example}

\DeclareMathOperator{\eig}{\varphi}
\DeclareMathOperator{\bbP}{\mathbb{P}}
\DeclareMathOperator{\bbN}{\mathbb{N}}
\DeclareMathOperator{\Cov}{Cov}
\newcommand{\proper}{\mathsf}

\newcommand{\pE}{\proper{E}}

\newcommand{\<}{\langle}
\renewcommand{\>}{\rangle}

\newcommand{\V}{{\bf \mathsf{V}}}
\DeclarePairedDelimiter\ceil{\lceil}{\rceil}

\renewcommand{\theenumi}{\normalfont{(\roman{enumi})}}

\endlocaldefs

\allowdisplaybreaks

\begin{document}

\begin{frontmatter}
		\title{Gaussian Whittle--Mat\'ern fields on metric graphs}

\begin{aug}
\author[A]{\fnms{David} \snm{Bolin}\ead[label=e1,mark]{david.bolin@kaust.edu.sa}} 
\author[A]{\fnms{Alexandre B.} \snm{Simas}\ead[label=e2]{alexandre.simas@kaust.edu.sa}} 
\author[B]{\fnms{Jonas} \snm{Wallin}\ead[label=e3]{jonas.wallin@stat.lu.se}}
\runauthor{David Bolin, Alexandre Simas and Jonas Wallin }
	\address[A]{Statistics Program, Computer, Electrical and Mathematical Sciences and Engineering Division, King Abdullah 
	University of 
	Science and Technology (KAUST), 
	\printead{e1}, \printead{e2}}
\address[B]{Department of Statistics,
	Lund University,
	\printead{e3}} 
\end{aug}

\begin{abstract}
	We define a new class of Gaussian processes on compact metric graphs such as street or river networks. 
	The proposed models, the Whittle--Mat\'ern fields, are defined via a fractional stochastic  
	differential equation on the compact metric graph and are a natural extension of Gaussian fields with 
	Mat\'ern covariance functions on Euclidean domains to the non-Euclidean metric graph setting. 
	Existence of the processes, as well as some of their main properties, such as sample path regularity are derived. The model 
	class in particular contains differentiable processes. To the best of our knowledge, 
	this is the first construction of a differentiable Gaussian process on general compact metric 
	graphs. Further, we prove an intrinsic property of these processes: that they do not change upon addition or removal of vertices with degree two.
	Finally, we obtain Karhunen--Lo\`eve expansions of the processes, provide numerical experiments, and compare them to Gaussian processes with isotropic covariance functions.
\end{abstract}



\end{frontmatter}

\section{Introduction}
In many areas of applications, statistical models need to be defined on networks such as connected rivers or street networks \cite{Hoef2006,okabe2012spatial,baddeley2017stationary,cronie2020}.  In this case, one wants to define the model using a metric on the network rather than the Euclidean distance between points. However, formulating Gaussian fields over linear networks, or more generally on metric graphs, is difficult. The reason being that it is hard to find flexible classes of functions that are positive definite under a non-Euclidean metric on the graph.

Often the shortest distance between two points is explored, i.e., the geodesic metric.  
A common alternative is the electrical resistance distance \cite{okabe2012spatial}, which was used recently by \cite{anderes2020isotropic} to create isotropic covariance functions on a subclass of metric graphs with Euclidean edges.
They in particular showed that for such graphs, one 
can define a valid Gaussian field by taking the covariance function to be of 
Mat\'ern type \cite{matern60}. That is, one can use the covariance of the form
\begin{equation}\label{eq:matern_cov}
r(s,t) = \frac{\Gamma(\nu)}{\tau^2\Gamma(\nu + \nicefrac{1}{2})\sqrt{4\pi}\kappa^{2\nu}}(\kappa d(s,t))^{\nu}K_\nu(\kappa d(s,t)),
\end{equation}
choosing $d(\cdot,\cdot)$ as the resistance metric. Here, 
$\tau,\kappa>0$ are parameters controlling the variance and practical correlation range, and 
$0<\nu\leq \nicefrac1{2}$ is a parameter controlling the sample path regularity. 
Further, $K_\nu(\cdot)$ is the modified Bessel function of the second kind and $\Gamma(\cdot)$ denotes the gamma function.  The restriction $\nu\leq \nicefrac1{2}$ means that we cannot use this approach to create differentiable Gaussian processes on metric graphs, even if they have Euclidean edges. 
Because of this and the many other difficulties with creating Gaussian fields via covariance functions on non-Euclidean spaces, we take a different approach in this work and focus on creating a Gaussian random field $u$ on a compact metric graph $\Gamma$ as a solution to a stochastic differential equation
\begin{equation}\label{eq:Matern_spde}
(\kappa^2 - \Delta)^{\nicefrac{\alpha}{2}} (\tau u) = \mathcal{W}, \qquad \mbox{on $\Gamma$},
\end{equation}
where $\alpha = \nu + \nicefrac1{2}$, $\Delta$ is  the Laplacian equipped with suitable 
``boundary conditions'' in the vertices, and $\mathcal{W}$ is Gaussian white noise. 
The advantage with this approach is that, if the solution exists, it automatically 
has a valid covariance function. The reason for considering this particular equation 
is that when \eqref{eq:Matern_spde} is considered on $\mathbb{R}^d$, it has 
Gaussian random fields with the covariance function \eqref{eq:matern_cov},
with $d(\cdot,\cdot)$ being the Euclidean distance,
as stationary solutions \cite{whittle63}. The method of generalizing the Mat\'ern fields to Riemannian manifolds by \emph{defining} Whittle--Mat\'ern fields as solutions to \eqref{eq:Matern_spde} specified on the manifold was proposed by \cite{lindgren11}, and has since then been extended to a number of scenarios \cite[see][for a recent review]{lindgren2022spde}, including non-stationary  \cite{bakka2019, Hildeman2020} and non-Gaussian \cite{bolin14,bw20} models. 

The difficulty in extending the SPDE approach to metric graphs is that it is not clear how to define the differential operator in this case, and it is also not clear what type of covariance functions one would obtain. We will in this work use quantum graph theory \cite{Berkolaiko2013} to define the operator and show that \eqref{eq:Matern_spde} then has a unique solution, for which we can derive sample path regularity properties.  
For $\alpha=1$ we obtain a field with a covariance function that is similar to the exponential covariance, i.e., the case $\nu=\nicefrac1{2}$ in \eqref{eq:matern_cov}, which was shown to be a valid covariance for metric graphs with Euclidean edges by \cite{anderes2020isotropic}. However, our construction has the advantage that it is well-defined for \emph{any} compact metric graph, which may contain loops and multiple edges between vertices, and not only for the subclass with Euclidean edges. 
Furthermore, the model is well-defined for any value of $\alpha>\nicefrac{1}{2}$, and we show that it yields a process with differentiable sample paths if $\alpha>\nicefrac{3}{2}$. In this case one cannot use the corresponding Mat\'ern covariance function to construct a valid Gaussian process, even for graphs with Euclidean edges \cite{anderes2020isotropic}. Thus, this construction provides, as far as we know, for the first time a covariance function for differentiable random fields on general compact metric graphs. See Figure~\ref{fig:cov_example} for an example. 

\begin{figure}[t]
\label{fig:cov_example}
\includegraphics[width=0.4\linewidth]{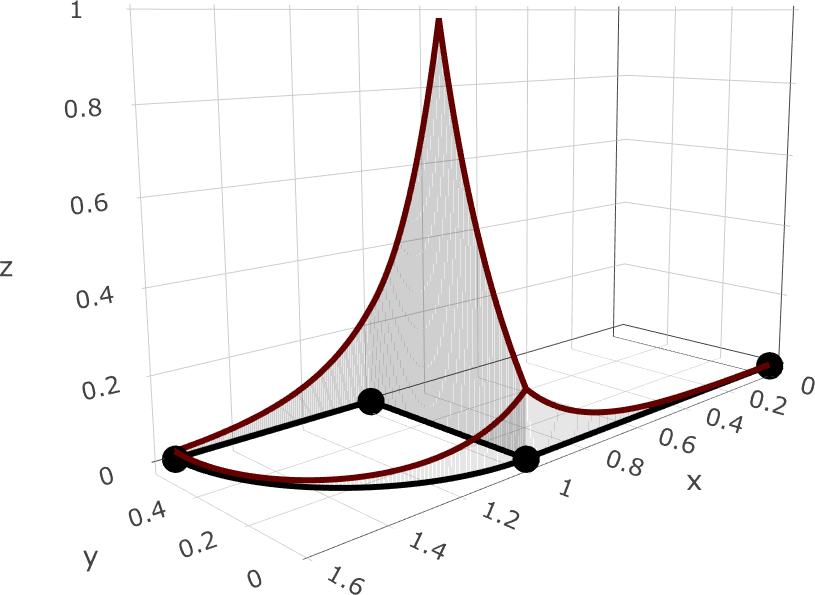}
\includegraphics[width=0.4\linewidth]{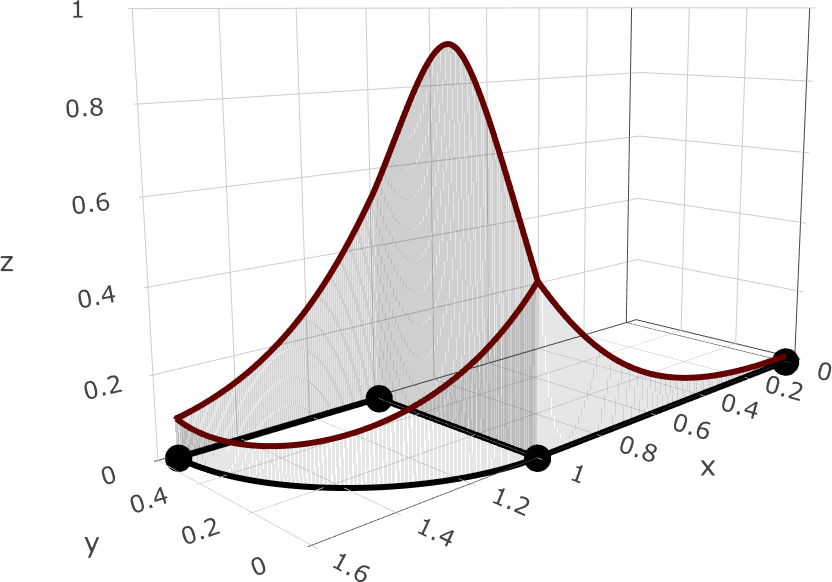}
\includegraphics[width=0.4\linewidth]{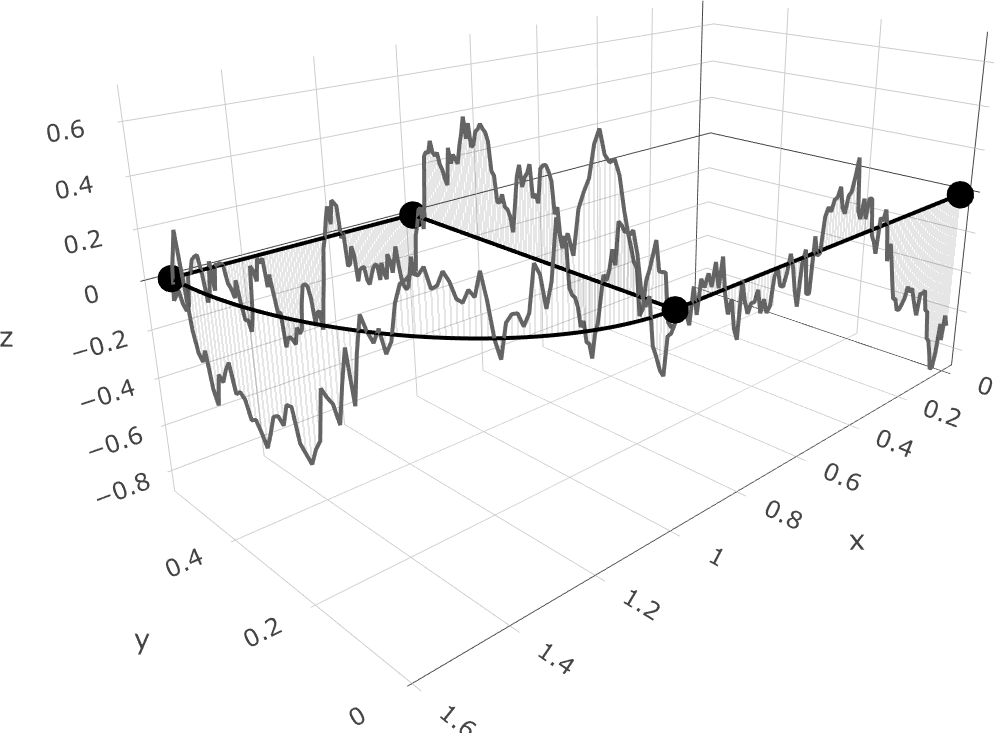}
\includegraphics[width=0.4\linewidth]{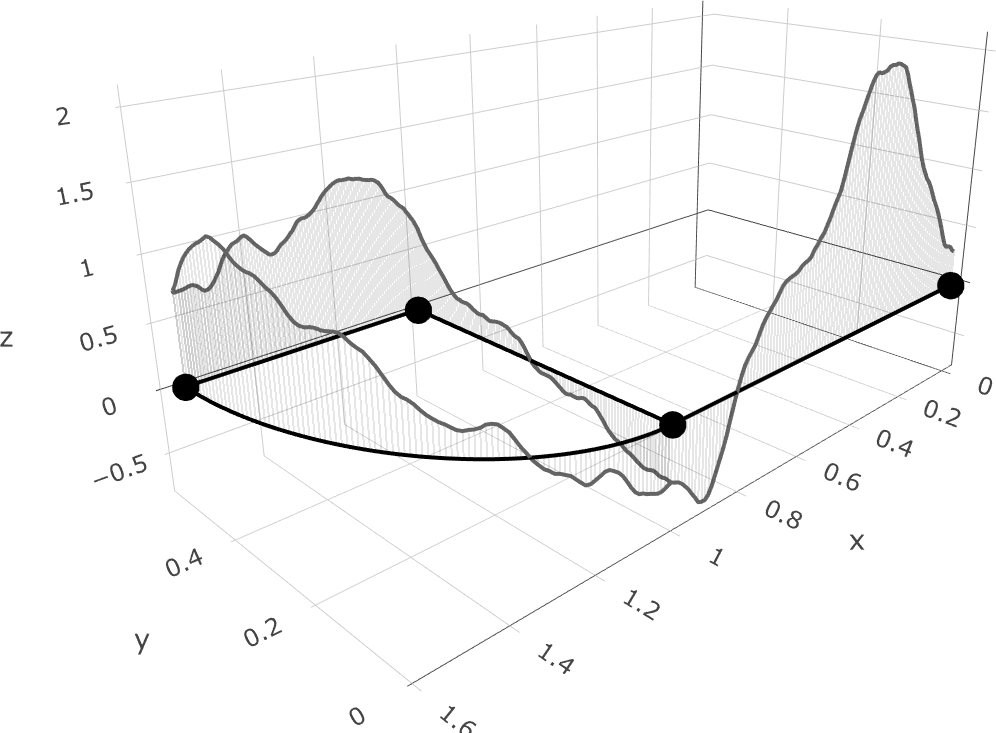}
\caption{Examples of covariance functions (top) and corresponding realizations (bottom) of a Whittle--Mat\'ern field with $\alpha=1$ (left) and $\alpha=2$ (right) on a graph. In both cases, $\kappa=5$ and the covariance $\Cov(u(s_0), u(s))$, where $s_0$ is the midpoint of the edge from $(x,y) = (1,0)$ to $(x,y) = (1,\nicefrac{1}{(1+\nicefrac{\pi}{4})})$,  is shown. }
\end{figure}

Defining Gaussian random fields on graphs has recently received some attention.  
In particular various methods based on the graph Laplacian have been proposed \cite{Alonso2021,dunson2020graph,borovitskiy2021matern}. 
However, these approaches do not define the Gaussian random field on the metric graph, but only on the vertices. 
Since it is unclear if methods based on the graph Laplacian can be used to define a Gaussian process on the entire metric graph, we will not consider these approaches in this work. 
Another important advantage of the proposed processes is that, as we will show, one can add vertices on existing edges, and/or remove vertices of degree two, in the graph without changing the process. 
Approaches based on the graph Laplacian lack this property, which means that their finite dimensional distributions are changed when adding observation locations to the graph, which clearly is not ideal for applications to data on metric graphs. For example, one cannot perform spatial prediction to unobserved locations without adding them as vertices in the graph prior to fitting the model, which would mean that one would change the model depending on whether one wants to perform prediction or not. Further, adding all prediction locations could become highly computationally intensive since it would increase the model size far beyond the size of the data. 

The outline of this work is as follows. In Section~\ref{sec:graphs}, we introduce the notion of compact metric graphs in detail and provide some of the key results from quantum graph theory which we will need in later sections. In Section~\ref{sec:spde} we introduce the Whittle--Mat\'ern fields via \eqref{eq:Matern_spde} and prove that the SPDE has a unique solution. 
Section~\ref{sec:properties} provides results on the sample path regularity of the corresponding Whittle--Mat\'ern fields. 
The proofs of these results are provided in Section~\ref{app:proofsregularitysection}.
Further regularity properties of the covariance function of the solution are derived in Section \ref{sec:spectralrepr}, which also contains the derivation of a spectral representation which can be used to simulate the fields. Section \ref{sec:mean-square} contains results on the mean-square differentiability of the fields and Section~\ref{sec:isotropic} contains a brief comparison to Gaussian fields with isotropic covariance functions. The article concludes with a discussion in Section~\ref{sec:discussion}.

\section{Quantum graphs and notation}\label{sec:graphs}
In this section, we introduce the notation that will be used throughout the article as well as some of the key concepts from quantum graph theory that we will need in later sections. We refer to \cite{Berkolaiko2013} for a complete introduction to the field.

\subsection{Compact metric graphs and notation}
A finite undirected metric graph $\Gamma$ consists of a finite set of vertices 
$\mathcal{V}=\{v_i\}$ and a set $\mathcal{E}=\{e_i\}$ of edges connecting 
the vertices. We write $u\sim v$ or $(u,v)\in\mathcal{E}$ if $u$ and $v$ 
are connected with an edge. Each edge $e$ is defined by a pair of vertices and a length $\ell_e \in (0,\infty)$. For every $v\in\mathcal{V}$, we denote the set of edges incident to the vertex $v$ by $\mathcal{E}_v$. We assume that the graph is connected, so that there is a path between all vertices, and  we assume that the degree of each vertex (i.e., the number of edges connected to it) is finite. Since the lengths $\ell_e$ are also finite,this means that the graph is compact. 
A natural choice of metric for the graph is the shortest path distance, which for any two points in $\Gamma$ is defined as the length of the shortest path in $\Gamma$ connecting the two. We denote this metric by $d(\cdot,\cdot)$ from now on.  
A point $s\in \Gamma$ is a position on some edge $e$, and can be represented as $s=(e,t)$ where $t\in[0,\ell_e]$. 

\begin{Remark}
A typical example of an edge is a simple (i.e., without self intersection) piecewise $C^1$ curve $\gamma:[0,\ell_e]\to \mathbb{R}^d$, $0<\ell_e<\infty$, for some $d\in\mathbb{N}$, which is regular (i.e., for 
all $t\in [0,\ell_e]$ such that $\gamma$ is differentiable at $t$, we have $\gamma'(t)\neq 0$) and is parameterized by arc-length. It is a basic result in differential geometry that every regular piecewise $C^1$ curve admits a parameterization given by arc-length. Therefore, if we take the edge $e$ to be induced by the curve $\gamma$, then given $t,t'\in [0,\ell_e]$, with $t<t'$, the distance between the points $s = (e,t)$ and $s'=(e,t')$ is the length of the curve $\gamma([t,t'])$, which is given by $L(\gamma|_{[t,t']}) = t'-t$. Thus, a cure of this type is a suitable choice for edges in compact metric graphs, since the curve under these assumptions is isometric to a closed interval.	
\end{Remark}

A function $f:\Gamma\to\mathbb{R}$ is a collection of functions $\{f_e\}$ defined on the edges $e\in\mathcal{E}$. For each edge $e\in\mathcal{E}$, $f|_e = f_e$, so we will always use the subscript $e$ to indicate the restriction of a function on $\Gamma$ to an edge $e$. Moreover, for each $e\in\mathcal{E}$, we will identify $e$ with its parameterization in the interval $[0,\ell_e]$, so we will always regard $f_e$ as a function defined on $[0,\ell_e]$.

We define $L_2(\Gamma)$ as the direct sum of the $L_2(e)$ spaces on the edges of $\Gamma$,
$L_2(\Gamma) = \bigoplus_{e \in \mathcal{E}} L_2(e)$, and equip it with the norm  $\|f\|_{L_2(\Gamma)}^2 =  \sum_{e\in\mathcal{E}}\|f_e\|_{L_2(e)}^2$.
That is, we have ${f = \{f_e\}_{e\in \mathcal{E}} \in L_2(\Gamma)}$ if ${f_e \in L_2(e)}$ for each $e\in \mathcal{E}$.
Here, $L_2(e)$ is defined with respect to the Lebesgue measure on $e$. 
We further let $C(\Gamma) = \{f\in L_2(\Gamma): f\hbox{ is continuous}\}$ denote the space of continuous functions on $\Gamma$  and define ${\|f\|_{C(\Gamma)} = \sup\{|f(s)|: s\in\Gamma\}}$ as the supremum norm of a function $f$. 
Similarly, for $0<\gamma<1$, 
we let $C^{0,\gamma}(\Gamma)$ denote the space of $\gamma$-H\"older continuous functions, that is, all $f\in C(\Gamma)$ with ${\|f\|_{C^{0,\gamma}(\Gamma)}<\infty}$. 
Here $\|f\|_{C^{0,\gamma}(\Gamma)} = \|f\|_{C(\Gamma)} + [f]_{C^{0,\gamma}(\Gamma)}$ is the $\gamma$-H\"older norm and
$$
[f]_{C^{0,\gamma}(\Gamma)} = \sup\left\{\frac{|f(s)-f(s')|}{d(s,s')^{\gamma}}: s,s'\in \Gamma, s\neq s' \right\},
$$
is the $\gamma$-H\"older seminorm.

We define the Sobolev space $H^1(\Gamma)$ as the space of all continuous functions $f\in C(\Gamma)$ that satisfy 
$
\|f\|_{H^1(\Gamma)}^2 = \sum_{e\in\mathcal{E}} \|f_e\|_{H^1(e)}^2 < \infty,
$
where $H^1(e)$ is the Sobolev space of order 1 on the edge $e$. That is, $H^1(\Gamma) = C(\Gamma)\cap \bigoplus_{e\in\mathcal{E}} H^1(e)$.
Given $f\in H^1(\Gamma)$, the weak derivative of $f$, denoted by $f'$, $Df$, $df/ds$, etc., 
is defined as the function $f'\in L_2(\Gamma)$ such that for every $e\in\mathcal{E}$, $f'|_{e}(t) = (f|_e)'(t)$, a.e., 
where $(f|_e)'$ is the weak derivative of $f|_e$, which is well-defined since $f|_e \in H^1(e)$.
Observe that $f'$ is the unique function in $L_2(\Gamma)$ such that 
$\forall \phi\in \bigoplus_{e\in\mathcal{E}} C^\infty_c(e)$,
$$
\sum_{e\in\mathcal{E}} \int_0^{\ell_e} f_e(t)\phi_e'(t) dt = -\sum_{e\in\mathcal{E}} \int_0^{\ell_e} f_e'(t) \phi_e(t) dt.
$$
Now, note that the continuity assumption on $H^1(\Gamma)$ guarantees that $f\in H^1(\Gamma)$ is uniquely defined at the vertices. 
Dropping the continuity requirement, one can construct decoupled Sobolev spaces of any order $k\in\mathbb{N}$  as the direct sum 
$\widetilde{H}^k(\Gamma) = \bigoplus_{e\in \mathcal{E}} H^k(e)$, which is
endowed with the Hilbertian norm $
{\|\cdot\|_{\widetilde{H}^k(\Gamma)}^2 = \sum_{e\in\mathcal{E}} \|\cdot\|_{H^k(e)}^2},
$ where $H^k(e)$ is the Sobolev space of order $k$ on the edge $e$. See Appendix~\ref{app:sob} 
for further details about these Sobolev spaces.
Finally, for $0<\alpha<1$, we define the fractional Sobolev space of order $\alpha$ by the real interpolation between the spaces $L_2(\Gamma)$ and $H^1(\Gamma)$ through the so-called $K$-method, 
$H^\alpha(\Gamma) = (L_2(\Gamma), H^1(\Gamma))_\alpha.$
See Appendix \ref{app:interpolation} and Appendix \ref{app:sob} 
for further details on real interpolation of Hilbert spaces.

\begin{Remark}
One should note that the choice of the $K$-method in $H^\alpha(\Gamma)$ is not relevant here.
Indeed, since $L_2(\Gamma)$ and $H^1(\Gamma)$ are Hilbert spaces, with $H^1(\Gamma)\subset L_2(\Gamma)$,
it follows that the $K$-method and the $J$-method of interpolations are equivalent, see for instance
\cite[Theorem 3.1]{chandlerwildeetal}. Actually, any geometric interpolation space will
yield equivalent fractional Sobolev spaces \cite[Theorem 3.5]{chandlerwildeetal}. In
particular, the complex interpolation method also yields equivalent fractional spaces
\cite[Remark 3.6]{chandlerwildeetal}.
\end{Remark}

We let $(\Omega, \mathcal{F},\mathbb{P})$ be a complete probability space, 
and for a real-valued random variable $Z$, we let $\pE(Z) = \int_\Omega Z(\omega) d\mathbb{P}(\omega)$ 
denote its expectation. Further, $L_2(\Omega)$ denotes the Hilbert space of all (equivalence classes of) 
real-valued random variables with finite second moment. Finally, for a separable Hilbert 
space $(E, \|\cdot \|_E)$, we let $L_2(\Omega; E)$ be the space of $E$-valued Bochner measurable random 
variables with finite second moment, which is equipped with the norm 
$\|u\|_{L_2(\Omega; E)}^2 = \int_\Omega \|u(\omega)\|_E^2 d\mathbb{P}(\omega)$. 

\subsection{Quantum graphs}\label{subsec:quantumgraphs}
A metric graph coupled with a differential operator on that graph is referred to as a quantum 
graph. The most important differential operator in this context is the Laplacian. However, 
there is no unique way of defining the Laplacian on a metric graph. For a function $f\in\widetilde{H}^2(\Gamma)$, the operator is naturally defined as the second derivative on each edge, but at the vertices there are several options of ``boundary conditions'' or \emph{vertex conditions}. One of the most popular choices is the Kirchhoff conditions
$
\left\{\mbox{$f $ is continuous on $\Gamma$ and $\forall v\in\mathcal{V}: \sum_{e \in\mathcal{E}_v} \partial_e f(v) = 0$} \right\},
$
where $\partial_e$ denotes the directional derivative away from the vertex. 
If $f\in\widetilde{H}^2(\Gamma)$, this implies that $f$ restricted to each edge 
is differentiable and the Kirchhoff conditions are therefore well-defined 
\cite{Arioli2017FEM}. 
The Laplacian with these vertex conditions is often denoted as the  Kirchhoff-Laplacian, 	 
which we here simply denote by $\Delta_{\Gamma}$. Let 
\begin{equation}\label{eq:Kspace}
K(\Gamma) = \Bigl\{f\in\widetilde{H}^2(\Gamma): \textstyle \sum_{e\in\mathcal{E}_v} \partial_e f(v) = 0\Bigr\},
\end{equation} 
then by \cite[Section~2.3]{kostenkoetal}, the domain of the Kirchhoff Laplacian $\Delta_\Gamma$ is 
$\mathcal{D}(\Delta_\Gamma) = \widetilde{H}^2(\Gamma) \cap C(\Gamma) \cap K(\Gamma)$, 
and  $\Delta_\Gamma : \mathcal{D}(\Delta_\Gamma) \rightarrow L_2(\Gamma)$ 
is defined as $\Delta_\Gamma := \oplus_{e\in \mathcal{E}}\Delta_e$, where $\Delta_e = \nicefrac{d^2}{d t^2}$ with $\mathcal{D}(\Delta_e) = H^2(e)$.

By \cite[][Theorem 1.4.4]{Berkolaiko2013}, $\Delta_\Gamma$ is self-adjoint. It turns out that this particular 
Laplacian is convenient for defining Whittle--Mat\'ern fields on metric graphs, and we from now on only consider 
this choice. A few reasons for this is that the Kirchhoff vertex conditions are the natural extension of Neumann 
boundary conditions to the graph setting, and they are, because of this, in fact, often referred to as 
Kirchhoff--Neumann, Neumann, or even ``standard'' boundary conditions \cite{Berkolaiko2013}. Further, 
they enforce continuity across the vertices without completely decoupling the problem as would be the case if homogeneous Dirichlet vertex conditions were used (i.e., assuming that $f(v)  = 0$ for each vertex $v$).

We let $\{\hat{\lambda}_i\}_{i\in\mathbb{N}}$ denote the eigenvalues of $\Delta_{\Gamma}$, sorted in non-decreasing order, and let $\{\eig_i\}_{i\in\mathbb{N}}$ denote the corresponding eigenfunctions. 
By Weyl's law \cite{Odzak2019Weyl} we have that $\hat{\lambda}_i \sim i^2$ as $i\rightarrow\infty$, so there exists constants $c_\lambda, C_\lambda$ such that $0<c_\lambda< C_\lambda<\infty$ and
\begin{equation}\label{eq:weyl}
\forall i\in\mathbb{N}, \quad c_\lambda i^2 \leq \hat{\lambda}_i \leq C_\lambda i^2.
\end{equation}

\section{Whittle--Mat\'ern fields on compact metric graphs}\label{sec:spde}
Let $\kappa^2>0$ and define the differential operator 
\begin{equation}\label{eq:Loperator}
Lu = (\kappa^2 - \Delta_\Gamma) u, \quad u\in \mathcal{D}(L) = \mathcal{D}(\Delta_\Gamma) \subset L_2(\Gamma).
\end{equation}
The operator induces a symmetric and continuous bilinear form
\begin{equation}\label{eq:bilinearform}
a_L : H^1(\Gamma) \times H^1(\Gamma) \rightarrow \mathbb{R}, \quad a_L(\phi,\psi) = (\kappa^2 \phi,\psi)_{L_2(\Gamma)} + \sum_{e\in\mathcal{E}}(\phi_e',\psi_e')_{L_2(e)},
\end{equation}
which is coercive with coercivity constant $\min(1,\kappa)$ \cite{Arioli2017FEM}. 
Furthermore, it is clear,
from our discussion in subsection \ref{subsec:quantumgraphs} together
with \eqref{eq:bilinearform}, that the operator is densely 
defined and positive definite. Since the metric graph is compact, $L$ has a discrete spectrum, where each eigenvalue has finite multiplicity \cite[][Chapter 3]{Berkolaiko2013}. Clearly, the operator $L$ diagonalizes with respect to the eigenfunctions of $\Delta_{\Gamma}$, and it has eigenvalues $\{ \lambda_i\}_{i\in\mathbb{N}} = \{\kappa^2 + \hat{\lambda}_i\}_{i\in\mathbb{N}}$. 
In order to define the Whittle--Mat\'ern fields, we introduce the fractional operator $L^{\nicefrac{\alpha}{2}}, \alpha>0$, in the spectral sense as follows \cite[see also][]{BKK2020}. 
Define the Hilbert space 
\begin{equation}\label{eq:Hdot_def}
\dot{H}^{\alpha} = \mathcal{D}(L^{\nicefrac{\alpha}{2}}) := \{ \phi \in L_2(\Gamma) : \|\phi\|_{\alpha} < \infty\},	
\end{equation}
with inner product $(\phi,\psi)_{\alpha} := (L^{\nicefrac{\alpha}{2}}\phi, L^{\nicefrac{\alpha}{2}} \psi)_{L_2(\Gamma)}$ and induced norm $\|\phi\|_{\alpha} = \|L^{\nicefrac{\alpha}{2}} \phi\|_{L_2(\Gamma)}$. Here, the action of $L^{\nicefrac{\alpha}{2}} : \mathcal{D}(L^{\nicefrac{\alpha}{2}}) \rightarrow L_2(\Gamma)$ is defined by
$$
L^{\nicefrac{\alpha}{2}} \phi := \sum_{i\in\mathbb{N}} \lambda_i^{\nicefrac{\alpha}{2}} (\phi,\eig_i)_{L_2(\Gamma)}\eig_i, \quad \phi \in \mathcal{D}(L^{\nicefrac{\alpha}{2}}),
$$
where $\{\varphi_i\}_{i\in\mathbb{N}}$ are the eigenfunctions of $L$, and thus $\|\phi\|_{\alpha}^2 =  \sum_{i\in\mathbb{N}} \lambda_i^{\alpha } (\phi,\eig_i)_{L_2(\Gamma)}^2$. We denote the dual space of $\dot{H}^{\alpha}$ by $\dot{H}^{-\alpha}$, which has norm $\|\phi\|_{-\alpha}^2 = \sum_{i\in\mathbb{N}} \lambda_i^{-\alpha} \langle\phi,\eig_i\rangle^2$, where $\langle\cdot,\cdot \rangle$ denotes the duality pairing between $\dot{H}^{-\alpha}$ and $\dot{H}^{\alpha}$.

Let $\mathcal{W}$ denote Gaussian white noise on $L_2(\Gamma)$, which we may, formally, represent  through the series expansion $\mathcal{W} = \sum_{i\in\mathbb{N}} \xi_i \eig_i$ $\mathbb{P}$-a.s., where $\{\xi_i\}_{i\in\mathbb{N}}$ are independent standard Gaussian random variables on $(\Omega, \mathcal{F},\mathbb{P})$. By \cite[][Proposition 2.3]{BKK2020} this series converges in $L_2(\Omega; \dot{H}^{-s})$ for any $s>\nicefrac12$, and thus, $\mathcal{W} \in \dot{H}^{-\nicefrac1{2}-\epsilon}$ holds $\mathbb{P}$-a.s.\ for any $\epsilon>0$. Alternatively, we can represent Gaussian white noise as a family of centered Gaussian variables $\{\mathcal{W}(h) : h\in L_2(\Gamma)\}$ which satisfy
\begin{equation}\label{isometryW}
\forall h,g\in L_2(\Gamma), \quad \pE[\mathcal{W}(h)\mathcal{W}(g)] = (h,g)_{L_2(\Gamma)}.
\end{equation}

We define the Whittle--Mat\'ern fields through the fractional-order equation
\begin{equation}\label{eq:spde}
L^{\nicefrac{\alpha}{2}} (\tau u) = \mathcal{W}, 
\end{equation}
where $\tau>0$ is a constant that controls the marginal variances of $u$. 

\begin{Proposition}\label{prp:existuniqsol}
Given that $\alpha>\nicefrac12$,  \eqref{eq:spde} has a unique solution $u\in L_2(\Gamma)$ ($\mathbb{P}$-a.s.). 
\end{Proposition}
\begin{proof}
Since the operator $L$ is densely defined, positive definite
and has a compact resolvent, the result follows directly from the fact that $\mathcal{W} \in \dot{H}^{-\nicefrac1{2}-\epsilon}$ holds $\mathbb{P}$-a.s.\ for any $\epsilon>0$, in combination with \cite[][Lemma~2.1]{BKK2020}.	
\end{proof}
The solution $u$ is a centered Gaussian random field satisfying
\begin{equation}\label{eq:solspde}
\forall \psi\in L_2(\Gamma),\quad (u,\psi)_{L_2(\Gamma)} = \mathcal{W}(\tau^{-1}L^{-\nicefrac{\alpha}{2}}\psi) \quad \mathbb{P}\mbox{-a.s.},
\end{equation}
and it has covariance operator $\mathcal{C} = \tau^{-2}L^{-\alpha}$ satisfying
$
(\mathcal{C}\phi,\psi)_{L_2(\Gamma)} = \pE[(u,\phi)_{L_2(\Gamma)}(u,\psi)_{L_2(\Gamma)}]
$ for all ${\phi,\psi \in L_2(\Gamma)}$. 
We let $\varrho$ denote the covariance function corresponding to $\mathcal{C}$, defined as 
\begin{equation}\label{eq:covfunc}
\varrho(s,s') = \pE(u(s)u(s')) \quad \mbox{a.e. in $\Gamma\times \Gamma$}.
\end{equation}
Indeed, by Proposition \ref{prp:existuniqsol} and \cite[Remark 2.4]{BKK2020}, we have that $u\in L_2(\Omega,L_2(\Gamma))$. Thus,
by Fubini's theorem, $\pE(u(s)^2)<\infty$ for almost every $s\in\Gamma$. Therefore, by Cauchy-Schwarz inequality,
\eqref{eq:covfunc} exists and is well-defined for almost every $(s,s')\in \Gamma\times\Gamma$. We also have that $\varrho$ is, indeed, the
kernel of the covariance operator. This is a well-known fact in the theory of $L_2$ processes, but for completeness we also prove this in  
Lemma~\ref{prp:covintegralrep} in Section \ref{sec:spectralrepr}.

A crucial property of the Whittle--Mat\'ern fields is that we can remove or add vertices of degree 2 to the graph without changing the model. This is important since we then may, for example, add vertices to the graph at observation locations without changing the model.
We formulate this as a proposition. 

\begin{Proposition}\label{prop:join}
Suppose that $\Gamma$ has a vertex $v_i$ of degree 2, connected to edges $e_k$ and $e_\ell$, and define $\widetilde{\Gamma}$ as the graph where $v_i$ has been removed and $e_k,e_\ell$ merged to one new edge $\widetilde{e}_k$. Let $u$ be the solution to \eqref{eq:spde} on $\Gamma$, and let $\widetilde{u}$ be the solution to  \eqref{eq:spde} on $\widetilde{\Gamma}$, then $u$ and $\widetilde{u}$ have the same covariance function.
\end{Proposition}

\begin{proof}
Let $a_L : H^1(\Gamma) \times H^1(\Gamma) \rightarrow \mathbb{R}$ denote the bilinear form corresponding to $L$ on $\Gamma$, and let $\widetilde{a}_L$ denote the corresponding bilinear form on $\widetilde{\Gamma}$. 
Note that $H^1(\Gamma) = H^1(\widetilde{\Gamma})$ due to the continuity requirement, and since the bilinear forms  coincide for all functions on ${H^1(\Gamma)= H^1(\widetilde{\Gamma})}$,  the 
Kirchhoff-Laplacians $\Delta_\Gamma$ and $\Delta_{\widetilde{\Gamma}}$ have the same eigenfunctions and eigenvalues. Therefore, the covariance functions of the random fields $u$ and $\widetilde{u}$ as defined in terms of their spectral 
expansions (see Proposition \ref{prp:mercercov}) coincide. 
\end{proof}

\section{Sample path regularity}\label{sec:properties}
The goal of this section is to characterize the regularity  of solutions $u$ to the SPDE \eqref{eq:spde}. We will derive two main results. The first, Theorem~\ref{regularity1}, provides sufficient conditions for $u$ to have, $\bbP$-a.s., continuous sample paths, and sufficient conditions for existence of the weak derivatives. 
In the second regularity result, Theorem \ref{thm:weakregularity}, we fine-tune the conditions and are then able to ensure continuity of the sample paths even when  $u$ does not belong to $\dot{H}^1$. Furthermore, we are able to obtain H\"older continuity, with an order depending on the fractional exponent $\alpha$ of \eqref{eq:spde}. 

To simplify the exposition, we postpone the proofs of all statements in this section to Section~\ref{app:proofsregularitysection}. However, one of the key ideas of the proofs is to use interpolation theory (see Appendix~\ref{app:interpolation}) 
to characterize the spaces $\dot{H}^\alpha$. Another main idea is to exploit a novel Sobolev embedding for compact metric graphs, namely Theorem \ref{thm:sobembedding}, which we introduce and prove in this paper. 
We begin with the statement of our first regularity 
result.

\begin{Theorem}\label{regularity1}
Let $u$ be the solution of \eqref{eq:spde}.  
Furthermore, let $\alpha^\ast = \ceil{\alpha - \nicefrac{1}{2}}  -1$,
where $\ceil{a}$ denotes smallest integer larger than or equal to $a\in\mathbb{R}$. Then, we have that 
\begin{enumerate}
	\item for every $\epsilon>0$, $u\in L_2(\Omega; \dot{H}^{\alpha-\nicefrac{1}{2}-\epsilon})$;
	\item $u\in H^1(\Gamma)$ $\bbP$-a.s. if and only if $\alpha > \nicefrac{3}{2}$;
	\item for $j=1,\ldots, \alpha^\ast$, $D^j u$ exists $\bbP$-a.s.;
	\item for $\alpha\geq 2$ and ${1\leq k \leq \alpha^\ast-1}$, 
	such that $k$ is an odd integer, we have that 
	${\sum_{e\in\mathcal{E}_v} \partial_e^k u(v) =0}$, $\bbP$-a.s., for every $v\in\mathcal{V}$, where $\partial_e^k u$ is the directional derivative of $D^{k-1}u$.
\end{enumerate}
\end{Theorem}

Recall that, for two Hilbert spaces $E, F$, we have the continuous embedding $E\hookrightarrow F$ if the inclusion map from $E$ to $F$ is continuous, i.e., there is a $C>0$ such that for every $f\in E$, $\|f\|_F \leq C \|f\|_E$.

\begin{Theorem}[Sobolev embedding for compact metric graphs]\label{thm:sobembedding}
Let $\nicefrac{1}{2} < \alpha \leq 1$ and $\Gamma$ be a compact metric graph. We have the continuous embedding 
$H^\alpha(\Gamma) \hookrightarrow C^{0,\alpha-\nicefrac{1}{2}}(\Gamma).$
\end{Theorem}

As a direct consequence we also have the embedding:

\begin{Corollary}\label{cor:sobembeddingHdot}
Let $\alpha>\nicefrac{1}{2}$ and define $\widetilde{\alpha} = \alpha - \nicefrac{1}{2}$ if $\alpha\leq 1$ and $\widetilde{\alpha}=\nicefrac{1}{2}$ if $\alpha>1$. Then
$\dot{H}^\alpha \hookrightarrow C^{0,\widetilde{\alpha}}(\Gamma).$
\end{Corollary}

By combining the Sobolev embedding with a suitable version of the Kolmogorov--Chentsov 
theorem, we arrive at the following regularity result, which shows that 
the smoothness parameter $\alpha$ provides precise control over   
$\gamma$-H\"older continuity of the sample paths and their weak derivatives. 

\begin{Theorem}\label{thm:weakregularity}
Fix  $\alpha > \nicefrac{1}{2}$ and let $\widetilde{\alpha} = \min\{\alpha-\nicefrac{1}{2}, \nicefrac{1}{2}\}$.
Let, also, $\lfloor a \rfloor$ denote the largest integer smaller than or equal to $a\in\mathbb{R}$. Then, we have that 
\begin{enumerate}
	\item for $0<\gamma<\widetilde{\alpha}$, the solution $u$ of \eqref{eq:spde} has a modification with $\gamma$-H\"older continuous sample paths;
	\item if $\lfloor\alpha\rfloor$ is even and $\alpha - \lfloor\alpha\rfloor > \nicefrac{1}{2}$, then for any $\tilde{\gamma}$ such that $0<\tilde{\gamma} < \alpha - \lfloor\alpha\rfloor - \nicefrac{1}{2}$ we have 
	that $D^{\lfloor\alpha\rfloor}u$ has a modification with $\tilde{\gamma}$-H\"older continuous sample paths.
	\item if $\alpha \geq 2$ and $\lfloor \alpha \rfloor$ is odd, then for every $0<\gamma < \min\{\alpha-\lfloor \alpha\rfloor,\nicefrac12\}$, we have that $D^{\lfloor \alpha\rfloor -1} u$ has
	a modification with $\gamma$-H\"older continuous sample paths.
	\item if $\alpha\geq 3$, then for
	every $0<\gamma<\nicefrac12$ and every $j=0,\ldots,\left\lfloor\nicefrac{(\alpha - 3)}{2}\right\rfloor$, the derivative $D^{2j}u$ has a modification with $\gamma$-H\"older continuous sample paths.
\end{enumerate}
\end{Theorem}

It is noteworthy that in Theorem \ref{regularity1} and \ref{thm:weakregularity}, we can only establish continuity of the even-order derivatives. The reason is that the odd-order derivatives 
satisfy the vertex condition ${\sum_{e\in\mathcal{E}_v} \partial_e^{2k+1} f(v) = 0}$, where $\partial_e^{2k+1}f$ is the directional derivative of $\partial^{2k}f$. 
This condition does not ensure continuity of $\partial^{2k+1} f$. Another problem is that in order to define continuity of $\partial^{2k+1}f$, one needs to fix, beforehand, a direction for each edge, 
as these derivatives are direction-dependent, which shows that there is no natural way of defining continuity of such functions. On the other hand, the even-order derivatives are direction-free, so their continuity 
is well-defined.
In Section~\ref{sec:mean-square}, we investigate the differentiability of the fields further by connecting the weak derivatives with  $L_2(\Omega)$ derivatives.

\section{Proofs of the results in Section~\ref{sec:properties}}\label{app:proofsregularitysection}
We start by introducing some additional notation that we need in this section. 
For two Hilbert spaces $E$ and $F$, we use the notation $E\cong F$ if $E \hookrightarrow F \hookrightarrow E$. The space of all bounded linear operators from $E$ to $F$ is denoted by $\mathcal{L}(E;F)$, which is a Banach space when equipped with the operator norm ${\|T\|_{\mathcal{L}(E;F)} = \sup_{f\in E\setminus\{0\}}\nicefrac{\|Tf\|_F}{\|f\|_E}}$. We let $E^*$ denote the dual of $E$, which contains all continuous linear functionals $f : E \rightarrow \mathbb{R}$. 
Throughout this section we will, in addition to $s$, use the letters $x$ and $y$ to denote points in $\Gamma$.

Note that the eigenfunctions $\{\eig_i\}_{i\in\mathbb{N}}$ of  $\Delta_\Gamma$ form an 
orthonormal basis of $L_2(\Gamma)$ and is an eigenbasis for $L$ defined in \eqref{eq:Loperator}. 
Also note that $\eig_i\in H^1(\Gamma)$ for every $i\in\mathbb{N}$, and that 
$a_L(\eig_i,\phi) = \lambda_i (\eig_i,\phi)_{L_2(\Gamma)}$, where $a_L$ is defined in 
\eqref{eq:bilinearform} and $\phi\in H^1(\Gamma)$.  Thus, for each $i\in\mathbb{N}$,
\begin{equation}\label{intbypartseigenvec}
(\eig_i', \phi')_{L_2(\Gamma)} = \sum_{e\in\mathcal{E}} (\eig_i'|_e,\phi'|_e)_{L_2(e)} = (\lambda_i-\kappa^2) (\eig_i,\phi)_{L_2(\Gamma)}.
\end{equation}
Now, suppose that $\lambda_i = \kappa^2$ for some
$i\in\mathbb{N}$. Then, $\|\eig_i'\|_{L_2(\Gamma)} = 0$ and $\eig_i$ is constant on each edge $e\in\mathcal{E}$. Since $\eig_i\in H^1(\Gamma)$, we have that $\eig_i$ is continuous and, therefore, $\eig_i$ is constant on $\Gamma$. Thus, we directly have that if $\eig_i$ is non-constant, which means that $\|\eig_i'\|_{L_2(\Gamma)}^2 > 0$, then $\lambda_i > \kappa^2$. This also tells us that the eigenspace associated to the eigenvalue $\lambda_i = \kappa^2$ has dimension 1.

In the proofs we will assume, without loss of generality, that $\tau=1$ to keep 
the notation simple. 
\begin{Proposition}\label{prp:hdot1}
We have the identification
$\dot{H}^1 \cong H^1(\Gamma)$.
\end{Proposition}
\begin{proof}
To establish the identity we will show that both spaces are the same energetic space, 
in the sense of \cite{zeidler}, with respect to the operator $L$ on $L_2(\Gamma)$. Recall 
that energetic space is the closure of $\mathcal{D}(L)$ in $L_2(\Gamma)$ with respect to the energetic 
inner product norm $\|\cdot \|_E^2=(L\cdot,\cdot)_{L_2(\Gamma)}$ (this norm is also
sometimes called the energy norm), for more details see \cite[Section 5.3]{zeidler}.
We start by establishing the result for $H^1(\Gamma)$. By \cite[Section 2.3]{kostenkoetal} we have  that 
$\mathcal{D}(L) = \widetilde{H}^2(\Gamma) \cap C(\Gamma) \cap K(\Gamma)$. Therefore, $\mathcal{D}(L)\subset H^1(\Gamma)$.
Observe that the energetic norm on $\mathcal{D}(L)$ is given by
$$
\|f\|_E^2 = a_L(f,f) = (Lf,f)_{L_2(\Gamma)}	= (\kappa^2 f,f)_{L_2(\Gamma)} + \sum_{e\in\mathcal{E}}(f_e',f_e')_{L_2(e)},
$$
which, since $\kappa^2>0$, is clearly equivalent to $\|\cdot\|_{H^1(\Gamma)}$. 
Now by \cite[Theorem~2.8 and Remark 3.1]{kostenkoetal} we have that $\mathcal{D}(L)$ is 
dense in $H^1(\Gamma)$ with respect to $\|\cdot\|_{H^1(\Gamma)}$. Thus, 
$H^1(\Gamma)$ is the energetic space with respect to $L$.
We will now show that $\dot{H}^1$ is the same energetic space.
Since $\Gamma$ is a compact metric graph (thus, we have finitely many edges of finite length), 
it follows by Rellich--Kondrachov theorem, \cite[see, e.g.,][]{evans2010partial}, 
that $\widetilde{H}^1(\Gamma)$ is compactly embedded in $L_2(\Gamma)$. Therefore, 
the embedding $H^1(\Gamma) \subset L_2(\Gamma)$ is compact, so the
energetic space is compactly embedded in $L_2(\Gamma)$. Thus by 
\cite[Example 4, p. 296]{zeidler} it follows that $\mathcal{D}(L^{\nicefrac12})=H^1(\Gamma)$.
Since $\mathcal{D}(L^{\nicefrac12}) = \dot{H}^1$, the result follows.
\end{proof}

\begin{Corollary}\label{cor:identificationHalpha}
We have the identification
$\dot{H}^\alpha \cong H^\alpha(\Gamma).$
\end{Corollary}

\begin{proof}
By definition we have that  
$H^\alpha(\Gamma) = (L_2(\Gamma),H^1(\Gamma))_\alpha$,
and, by \cite[Theorem~4.36]{lunardi}, we have that 
$\dot{H}^\alpha \cong (\dot{H}^0, \dot{H}^1)_\alpha  = (L_2(\Gamma), \dot{H}^1)_\alpha$.
By Proposition \ref{prp:hdot1} it follows that $\dot{H}^1 \cong H^1(\Gamma)$, which by 
Lemma~\ref{lemma:interpolationEquivalence} 
gives 
$(L_2(\Gamma),H^1(\Gamma))_\alpha \cong (L_2(\Gamma), \dot{H}^1)_\alpha$.
\end{proof}

We will now obtain an embedding for higher order spaces. To this end, begin by
recalling that $\dot{H}^\alpha$ is the Hilbert space defined by \eqref{eq:Hdot_def}.

\begin{Proposition}\label{prp:hdotk}
We have 
$\dot{H}^2 = \widetilde{H}^2(\Gamma) \cap C(\Gamma) \cap K(\Gamma),$ where
$K(\Gamma)$ is given in \eqref{eq:Kspace}, and for integer $k > 2$,
\begin{equation}\label{eq:Hkset}
	\dot{H}^k \cong \left\{f\in \widetilde{H}^k(\Gamma): D^{2\left\lfloor \nicefrac{k}{2}\right\rfloor} f \in C(\Gamma), \,\forall m = 0,\ldots, \left\lfloor\nicefrac{(k-2)}{2} \right\rfloor, D^{2m} f\in\dot{H}^2 \right\}.
\end{equation}
Furthermore, the norms $\|\cdot\|_k$ and $\|\cdot\|_{\widetilde{H}^k(\Gamma)}$ are equivalent.
\end{Proposition}

\begin{proof}
The equality of the sets $\dot{H}^2 = \widetilde{H}^2(\Gamma) \cap C(\Gamma) \cap K(\Gamma),$
follows from  \cite[Section 2.3]{kostenkoetal}. Let us now show the equality of the sets in \eqref{eq:Hkset} for $k\geq 3$.
Note that
$f\in\dot{H}^3$ if and only if $f \in\dot{H}^2$ and $Lf \in \dot{H}^1.$
Since {$Lf = \kappa^2 f - \Delta_\Gamma f = \kappa^2 f - D^2f$} and $\dot{H}^2\subset\dot{H}^1$, we have 
\begin{align*}
	f\in \dot{H}^3 &\Leftrightarrow	f\in\dot{H}^2\hbox{ and }Lf \in \dot{H}^1 \Leftrightarrow f\in \dot{H}^2\hbox{ and }D^2 f\in \dot{H}^1\\
	& \Leftrightarrow f\in\dot{H}^2, D^2f \in \widetilde{H}^1(\Gamma), D^2f\in C(\Gamma) \Leftrightarrow f\in \widetilde{H}^3(\Gamma), f\in\dot{H}^2, D^2f\in C(\Gamma).
\end{align*}
The equality of the sets for $\dot{H}^k$, $k>3$, follows similarly.

It remains to prove the equivalence of the norms of $\dot{H}^k$ 
and $\widetilde{H}^k(\Gamma)$ for functions on $\dot{H}^k$. Recall that 
$(\lambda_{i})_{i\in\mathbb{N}}$ are the eigenvalues of $L$ and let
$$
\widehat{H}^k(\Gamma) = \left\{f\in \widetilde{H}^k(\Gamma): D^{2\left\lfloor \nicefrac{k}{2}\right\rfloor} f \in C(\Gamma), \,\forall m=0,\ldots, \left\lfloor\nicefrac{(k-2)}{2} \right\rfloor, D^{2m} f\in\dot{H}^2 \right\}.
$$
Since $\lambda_{j}\to \infty$ as $j\to\infty$, we have 
that $(\dot{H}^k, \|\cdot\|_{\dot{H}^k}) 
\hookrightarrow (L_2(\Gamma) , \|\cdot\|_{L_2(\Gamma)}).$ Denote by $I$ the 
inclusion map $I:\dot{H}^k \to \widetilde{H}^k(\Gamma)$.  
If $\phi_N \to 0$ in $\dot{H}^k$, then $\phi_N\to 0$ 
in $L_2(\Gamma)$. On the other hand if $I(\phi_N) \to \phi$, then 
$\|\phi_N - \phi\|_{L_2(\Gamma)}\leq \|\phi_N-\phi\|_{\widetilde{H}^k(\Gamma)} 
\to 0$. So, $\phi = 0$, since $\phi_N \to 0$ in $L_2(\Gamma)$. By the closed 
graph theorem $I$ is a bounded operator.
Conversely, by repeating the same argument from the previous inclusion, 
we obtain that $(\widehat{H}^k(\Gamma) , \|\cdot\|_{\widetilde{H}^k(\Gamma)}) 
\hookrightarrow (\dot{H}^k, \|\cdot\|_{\dot{H}^k})$ 
from the closed graph theorem. This proves the equivalence of norms.
\end{proof}

Let $\phi\in\dot{H}^k, k\in\mathbb{N}$, and for $i\in\mathbb{N}$, let $\alpha_i = (\phi,\eig_i)_{L_2(\Gamma)}$. Then, for every $m\in\mathbb{N}$, we have
\begin{equation}\label{evenorderdiff}
D^{2m}\phi = \sum_{i\in\mathbb{N}} \alpha_i \hat{\lambda}_i^{m} \eig_i=\sum_{i\in\mathbb{N}} \alpha_i (\lambda_i-\kappa^2)^{m} \eig_i.
\end{equation}
The above expression allows us to obtain additional regularity for the even-order derivatives of functions belonging to $\dot{H}^k$.

\begin{Corollary}\label{cor:evendiffcont}
Let $\phi \in \dot{H}^k$ for $k\in\mathbb{N}$. Then, $D^{2m}\phi\in H^1(\Gamma)$ for $m=0,\ldots, \left\lfloor\nicefrac{(k-1)}{2} \right\rfloor$. In particular, $D^{2m}\phi$ is continuous.
\end{Corollary}

\begin{proof}
Since, by Proposition \ref{prp:hdot1}, $\dot{H}^1\cong H^1(\Gamma)$, it is enough to show that ${D^{2m}\phi \in \dot{H}^1}$. Now, observe that as $\phi\in\dot{H}^k$, we have 
$\phi = \sum_{i\in\mathbb{N}} \alpha_i \eig_i$, with $\alpha_i = (\phi,\eig_i)_{L_2(\Gamma)}$, and $\sum_{i\in\mathbb{N}} \alpha_i^2\lambda_i^k<\infty$. By assumption, 
$m\leq (k-1)/2$, hence $2m + 1 \leq k$, which implies that $\sum_{i\in\mathbb{N}} \alpha_i^2 \lambda_i^{2m+1}<\infty$, since $\lambda_i\to\infty$ as $i\to\infty$. On the other hand, by \eqref{evenorderdiff}, 
$D^{2m}\phi = \sum_{i\in\mathbb{N}} \alpha_i \hat{\lambda}_i^{m} \eig_i.$ Since for every $i\in\mathbb{N}$, $\hat{\lambda}_i\leq \lambda_i$, we have that $D^{2m}\phi \in \dot{H}^1$, which concludes the proof.
\end{proof}

Now, before proving Theorem \ref{regularity1}, 
note that the derivatives that we will consider are well-defined:

\begin{Remark}
	For $\alpha\geq 1$ we have that, for any $\epsilon>0$, the realizations of $u$ belong to $\dot{H}^{\alpha-\nicefrac{1}{2}-\epsilon}$. 
	By Proposition \ref{prp:hdotk}, it follows for $\epsilon<\nicefrac{1}{2}$ and any $j \in \{0,\ldots,\left\lfloor\nicefrac{(\alpha-1)}{2}\right\rfloor\}$ that $\dot{H}^{\alpha-\nicefrac{1}{2}-\epsilon}\subset \dot{H}^{2j} \subset \widetilde{H}^{2j}(\Gamma)$ and, therefore, $D^{2j}u$ is well-defined. By an analogous argument, we have that $D^{\lfloor\alpha\rfloor}u$ is well-defined if $\lfloor\alpha\rfloor$ is even and $\alpha - \lfloor\alpha\rfloor > \nicefrac{1}{2}$.
\end{Remark}

\begin{Remark}
For $\alpha\geq 1$ and $j=0,\ldots, \left\lfloor\nicefrac{(\alpha-1)}{2}\right\rfloor$, we have that, for any $\epsilon>0$, the realizations of $u$ belong to $\dot{H}^{\alpha-\nicefrac{1}{2}-\epsilon}$. Let $j\leq \nicefrac{\alpha}{2} - \nicefrac{1}{2}$, which implies that $2j+1\leq \alpha$. Now, take $0 <\epsilon < \nicefrac{1}{2}$ so that ${2j \leq 2j + \nicefrac{1}{2} - \epsilon \leq \alpha - \nicefrac{1}{2} -\epsilon}$. Then, by Proposition \ref{prp:hdotk}, $\dot{H}^{\alpha-\nicefrac{1}{2}-\epsilon}\subset \dot{H}^{2j} \subset \widetilde{H}^{2j}(\Gamma)$ and, therefore, $D^{2j}u$ is well-defined for $j=0,\ldots,\left\lfloor\nicefrac{(\alpha-1)}{2}\right\rfloor$. By an analogous argument, we have that $D^{\lfloor\alpha\rfloor}u$ is well-defined if $\lfloor\alpha\rfloor$ is even and $\alpha - \lfloor\alpha\rfloor > \nicefrac{1}{2}$.
\end{Remark}

\begin{proof}[Proof of Theorem \ref{regularity1}]
First observe that by \cite[][Remark 2.4]{BKK2020}, for any $\epsilon>0$, $u\in L_2(\Omega; \dot{H}^{\alpha-\nicefrac{1}{2}-\epsilon})$, so that $u$ has sample paths in $\dot{H}^{\alpha-\frac{1}{2}-\epsilon}$. For $\alpha > \nicefrac{3}{2}$, we have directly that for sufficiently small $\epsilon>0$, $\dot{H}^{\alpha - \nicefrac{1}{2}-\epsilon}\subset \dot{H}^1$. By Proposition~\ref{prp:hdot1}, we have $\dot{H}^1\subset H^1(\Gamma)$. So, $\mathbb{P}$-a.s., $u$ belongs to $H^1(\Gamma)$. 
Furthermore, by Proposition~\ref{prp:hdotk}, 
$\dot{H}^{\alpha - \nicefrac{1}{2}-\epsilon} \subset 
\dot{H}^{\alpha^\ast}\subset \widetilde{H}^{\alpha^\ast}(\Gamma)$. 
This shows that for $j=1,\ldots,\alpha^\ast$, $\mathbb{P}$-a.s., 
$D^ju$ exists. 
The statement for odd-order derivatives follows directly from Proposition \ref{prp:hdotk}.

Finally, recall that $\mathcal{W} = \sum_{i\in\mathbb{N}} \xi_i \eig_i$ $\mathbb{P}$-a.s., where $\{\xi_i\}_{i\in\mathbb{N}}$ are independent standard Gaussian random variables on $(\Omega, \mathcal{F},\mathbb{P})$. By \eqref{eq:solspde}, we therefore have that ${(u,\eig_i)_{L_2(\Gamma)} = \lambda_i^{-\nicefrac{\alpha}{2}} \xi_i}$. Hence, if $\nicefrac12 < \alpha \leq \nicefrac32$, we have by \eqref{eq:weyl}, that there exist constants $K_1,K_2>0$, that do not depend on $i$, such that
$$
\|u\|_{1}^2 = \sum_{i\in\mathbb{N}} \lambda_i^{-\alpha+1} \xi_i^2 \geq K_1 \sum_{i\in\mathbb{N}} \lambda_i^{-\nicefrac12} \xi_i^2 \geq K_2 \sum_{i\in\mathbb{N}} i^{-1} \xi_i^2  = K_2 \sum_{i\in\mathbb{N}} \frac{1}{i} + 
K_2 \sum_{i\in\mathbb{N}} \frac{\xi_i^2-1}{i}.
$$
Since $\V\left[\sum_{i\in\mathbb{N}} \nicefrac{(\xi_i^2-1)}{i} 
\right]< \infty$, the second term on the right is finite $\mathbb{P}$-a.s. and it therefore follows that  $\|u\|_1 = \infty$ $\mathbb{P}$-a.s.
\end{proof}

Due to the complicated geometry of $\Gamma$ we need the following version of the  
Kolmogorov--Chentsov theorem to prove Theorem~\ref{thm:weakregularity}.

\begin{Theorem}\label{kolmchent}
Let $u$ be a random field on a compact metric graph $\Gamma$, and $M,p>0$ and $q>1$ be such~that
$$\pE(|u(x)-u(y)|^p) \leq M d(x,y)^q,\quad x,y\in\Gamma.$$
Then, for any $\beta \in (0,(q-1)/p)$, $u$ has a modification with $\beta$-H\"older continuous sample paths.
\end{Theorem}
\begin{proof}
Our goal is to apply \cite[Theorem 1.1]{kratschmerurusov}. To this end, we first observe 
that since $\Gamma$ is compact, it is totally bounded. Now, let $N(\Gamma,\eta)$ be the 
minimal number to cover $\Gamma$ with balls of radius $\eta>0$. Let $g$ be the maximum 
degree of the vertices in $\Gamma$. It is easy to see that we have
$$
N(\Gamma,\eta) \leq \frac{2\hbox{diam}(\Gamma) g}{\eta},
$$
where $\hbox{diam}(\Gamma) =  \sup\{d(x,y): x,y\in\Gamma\}$ is the diameter of the graph 
$\Gamma$.
This shows that we can apply \cite[Theorem 1.1]{kratschmerurusov} to any compact graph 
$\Gamma$, and that we can take $t=1$ in their statement. Therefore, if $u$ is a random 
field such that 
$\pE(|u(x)-u(y)|^p) \leq M d(x,y)^q$ for any $x,y\in\Gamma$, 
then, for any ${\beta \in (0,(q-1)/p)}$, $(u_x)_{x\in\Gamma}$ has a modification with 
$\beta$-H\"older continuous sample paths.
\end{proof}

Our goal now is to prove the Sobolev embedding for compact metric graphs. To this end, let us begin by 
introducing some notation. Given $x,y\in\Gamma$, let $[x,y]\subset \Gamma$ be any path connecting $x$ and $y$ with shortest length and for a function $u:\Gamma\to\mathbb{R}$, we denote by $u|_{[x,y]}$ the 
restriction of $u$ to the path $[x,y]$. 

We begin by showing that restrictions of functions in the Sobolev space $H^1(\Gamma)$ to a path $[x,y]$ belong to the Sobolev space $H^1([x,y])$, which is isometrically isomorphic to $H^1(I_{x,y})$, where ${I_{x,y} = \left[0,d(x,y)\right]}$ is an interval.

\begin{Proposition}\label{prp:restSobSpace}
Let $\Gamma$ be a compact metric graph. Let, also, $x,y\in\Gamma$, $x\neq y$. If $u\in H^1(\Gamma)$, then we have that $u|_{[x,y]}\in H^1([x,y])$.
\end{Proposition}

\begin{proof}
Let $[x,y] = p_1\cup p_2\cup\cdots\cup p_N$, for some $N\in\mathbb{N}$, where $p_i\subseteq e_i$, with $e_i\in\mathcal{E}$ and $e_i\neq e_j$ if $i\neq j$. Furthermore, let the paths be ordered in such a way that $p_1=[x,v_1]$, $p_N = [v_{N-1},y]$, ${p_i\cap p_{i+1} = \{v_i\}}$ and $p_i\cap p_j=\emptyset$ if $|i-j|>1$. Let $v_0 = x$ and $v_{N} = y$. By the definition of $H^1(\Gamma)$, we have that ${u|_{p_i} \in H^1(p_i)}$ for every $i$. Thus, by \eqref{eq:sobabscont}, 
we have that for every $z\in p_i$,
$u(z) = u(v_{i-1}) + \int_{[v_{i-1},z]} u'(t) dt.$
So, we have for each $z\in p_i$,
$$
u(z) = u(v_{i-1}) + \int_{[v_{i-1},z]} u'(s) ds = u(x) + \sum_{j=0}^{i-2} \int_{[v_j,v_{j+1}]} u'(s)ds + \int_{[v_{i-1},z]} u'(s) ds,
$$
by the continuity requirement of $u\in H^1(\Gamma)$. 
It then follows that
$u(z) = u(x) + \int_{[x,z]} u'(s)ds$
for every $z\in [x,y]$.
Therefore, by \eqref{eq:sobabscont}, 
$u\in H^1([x,y])$.
\end{proof}

We will now show that the restriction map $u\mapsto u|_{[x,y]}$ is a bounded linear operator from $H^\alpha(\Gamma)$ to $H^\alpha([x,y]) = (L_2([x,y]), H^1([x,y]))_\alpha$:

\begin{Proposition}\label{prp:restrboundedFrac}
Let $\Gamma$ be a compact metric graph, $x,y\in\Gamma$, $u\in L_2(\Gamma)$ and given a shortest path $[x,y]$, define ${T^{x,y}(u) = u|_{[x,y]}}$. Then, for $0<\alpha \leq 1$, we have that 
$T^{x,y}|_{H^\alpha(\Gamma)}$ is a bounded linear operator from $H^\alpha(\Gamma)$ to $H^\alpha([x,y])$.
\end{Proposition}

\begin{proof}
In view of Theorem \ref{thm:interpolationpairKmethod}, 
it is enough to show that the operator 
$T^{x,y}:L_2(\Gamma)\to L_2([x,y])$ is a couple map from $(L_2(\Gamma),H^1(\Gamma))$ to $(L_2([x,y]),H^1([x,y]))$.
We directly have  
${T^{x,y}(L_2(\Gamma)) \subset L_2([x,y])}$, and for every $u\in L_2(\Gamma)$, $\|u\|_{L_2([x,y])}\leq \|u\|_{L_2(\Gamma)}$, so that $T^{x,y}:L_2(\Gamma)\to L_2([x,y])$ is bounded. By Proposition~\ref{prp:restSobSpace}, $T^{x,y}(H^1(\Gamma))\subset H^1([x,y])$. It is clear that $\|u\|_{H^1([x,y])}\leq \|u\|_{H^1(\Gamma)}$ for every $u\in H^1(\Gamma)$, so that $T^{x,y}|_{H^1(\Gamma)}\to H^1([x,y])$ is bounded. Therefore, $T^{x,y}$ is a couple map from $(L_2(\Gamma),H^1(\Gamma))$ to $(L_2([x,y]),H^1([x,y]))$. The result thus follows from 
Theorem \ref{thm:interpolationpairKmethod}.
\end{proof}

\begin{proof}[Proof of Theorem \ref{thm:sobembedding}]
The proof is easier if $\Gamma$ has the following property: for any $x,y\in\Gamma$, there are vertices $v$ and $\widetilde{v}$ such that either $x,y$ belong to a single shortest path $[v,\widetilde{v}]$, or $x$ and $y$ belong to two different shortest paths connecting $v$ and $\widetilde{v}$. Observe that if we do not have loops, this is always true. However, if we have loops, this might not be true. Therefore, we will modify $\Gamma$ by adding vertices of degree 2 into a new metric graph $\widetilde{\Gamma}$, in such a way that $\widetilde{\Gamma}$ has this property.
Let $\mathcal{L}$ be the set of loops in $\Gamma$, where each $l\in\mathcal{L}$ consists of the edges and vertices of the loop. If $\mathcal{L}$ is empty, set $\widetilde{\Gamma} = \Gamma$. Otherwise, given a loop $l\in\mathcal{L}$, take a vertex $v$ from $l$. Now, let $p$ be the unique point in an edge in $L$ such that there exist two shortest paths from $v$ to $p$. If $p$ is already a vertex of $l$, there is nothing to do, otherwise, we can add $p$ to the graph as a vertex of degree 2. We repeat it for all the loops. Let $\widetilde{\Gamma}$ be the resulting graph. Now, by Proposition \ref{prop:join}, $H^1(\Gamma) = H^1(\widetilde{\Gamma})$. Thus, for $0\leq \alpha\leq 1$, $H^\alpha(\Gamma) = H^\alpha(\widetilde{\Gamma})$. Therefore, there is no loss in generality in taking $\Gamma$ as $\widetilde{\Gamma}$. 

Now, given $x,y\in\Gamma$, let $\mathcal{P}_{x,y}$ be the set containing the shortest paths connecting $x$ to $y$. Since $\Gamma$ is compact, $\mathcal{P}_{x,y}$ is finite. Take any $[x,y]\in \mathcal{P}_{x,y}$, and let 
$I_H^{x,y}: H^\alpha([x,y])\to C^{0,\alpha-\nicefrac{1}{2}}([x,y])$
be the Sobolev embedding on $[x,y]$, which, by Corollary \ref{cor:SobembedInterpFracSob}, 
is a bounded linear operator. Let, now, $T^{x,y}$ be the restriction operator, $T^{x,y}(u) = u|_{[x,y]}$, where $u\in L_2(\Gamma)$. By Proposition \ref{prp:restrboundedFrac}, we have that the composition $I_H^{x,y} \circ T^{x,y}$ is a bounded linear map from $H^\alpha(\Gamma)$ to $C^{0,\alpha-\nicefrac{1}{2}}([x,y])$. In particular, there exists $0< C_{x,y}<\infty$ such that for every $u\in H^\alpha(\Gamma)$,
\begin{equation}\label{eq:holderboundpath}
	\|u\|_{C^{0,\alpha-\frac{1}{2}}([x,y])} = \sup_{s\in [x,y]} |u(s)| + \sup_{\substack{s,s'\in [x,y]\\ s\neq s'}} \frac{|u(s)-u(s')|}{d(s,s')^{\alpha-\frac{1}{2}}} \leq C_{x,y} \|u\|_{H^\alpha(\Gamma)},
\end{equation}
where $C_{x,y}$ only depends on the length of the path and not on the particular path. 	
Now, note that for every $x,y\in\Gamma$, $x\neq y$, there exist $v,\widetilde{v}\in\mathcal{V}$ such that $[x,y] \subset \bigcup_{[v,\widetilde{v}]\in\mathcal{P}_{v,\widetilde{v}}} [v,\widetilde{v}]$. If, $[x,y]\subset [v,\widetilde{v}]$ for some vertices $v$ and $\widetilde{v}$, then, we directly have by \eqref{eq:holderboundpath} that
$$\frac{|u(x)-u(y)|}{d(x,y)^{\alpha-\frac{1}{2}}} \leq C_{v,\widetilde{v}}  \|u\|_{H^\alpha(\Gamma)}.$$	
On the other hand, if $x$ and $y$ belong to different shortest paths connecting $v$ to $\widetilde{v}$, then by our construction above, there is a vertex $w\in\{v,\widetilde{v}\}$ such that $d(x,w)\leq d(x,y)$ and $d(w,y)\leq d(x,y)$. Furthermore, 
\begin{equation}\label{eq:ineqSobEmbed}
	\frac{|u(x)-u(y)|}{d(x,y)^{\alpha-1/2}} \leq \frac{|u(x)-u(w)|}{d(x,y)^{\alpha-1/2}} + \frac{|u(w)-u(y)|}{d(x,y)^{\alpha-1/2}}  \leq \frac{|u(x)-u(w)|}{d(x,w)^{\alpha-1/2}} + \frac{|u(w)-u(y)|}{d(w,y)^{\alpha-1/2}}.
\end{equation}
Thus, by \eqref{eq:holderboundpath} and \eqref{eq:ineqSobEmbed},
$$
\frac{|u(x)-u(y)|}{d(x,y)^{\alpha-\frac{1}{2}}} \leq  2C_{v,\widetilde{v}}  \|u\|_{H^\alpha(\Gamma)} \leq  \|u\|_{H^\alpha(\Gamma)} \sum_{w\neq\widetilde{w}\in\mathcal{V}} 2C_{w,\widetilde{w}}.
$$
Since $\Gamma$ is compact, the set $\mathcal{V}$ is finite, so that
$C :=  2 \sum_{w\neq\widetilde{w}\in\mathcal{V}} C_{w,\widetilde{w}} < \infty.$
Therefore,
$$
[u]_{C^{0,\alpha-\frac{1}{2}}(\Gamma)} = \sup_{\substack{x,y\in\Gamma\\ x\neq y}} \frac{|u(x)-u(y)|}{d(x,y)^{\alpha-\frac{1}{2}}} \leq C \|u\|_{H^\alpha(\Gamma)}.
$$
Similarly, $\|u\|_{C(\Gamma)} \leq C \|u\|_{H^\alpha(\Gamma)}$ and we hence have that 
$\|u\|_{C^{0,\alpha-\nicefrac{1}{2}}}(\Gamma) \leq 2C \|u\|_{H^\alpha(\Gamma)}.$
\end{proof}

\begin{proof}[Proof of Corollary~\ref{cor:sobembeddingHdot} ]
Note that if $\alpha>1$, then $\dot{H}^\alpha\hookrightarrow \dot{H}^1$, so it is enough to consider ${\nicefrac12<\alpha\leq 1}$. The result thus follows directly from Theorem \ref{thm:sobembedding}, Proposition \ref{prp:hdot1} and Corollary \ref{cor:identificationHalpha}.
\end{proof}

We are now in a position to start to prove the more refined regularity statement. For this part, our ideas are inspired by the regularity results in \cite{coxkirchner}.
We begin with the following key result, which allows us to obtain a continuous process which we will show  solves \eqref{eq:spde}.

\begin{Lemma}\label{kolmchentbounds}
Fix $\alpha > \nicefrac12$ and let $\widetilde{\alpha} = \alpha-\nicefrac12$ if $\alpha\leq 1$ and $\widetilde{\alpha}=\nicefrac12$ if $\alpha>1$. For a location $s\in\Gamma$, define $u_0(s) = \mathcal{W}\left((\tau^{-1}L^{-\nicefrac{\alpha}{2}})^\ast(\delta_s)\right)$, where $\delta_s$ is the Dirac measure concentrated at $s\in\Gamma$ . Then,
\begin{equation}\label{holdercontvar}
	\pE\left(|u_0(x) - u_0(y)|^2\right) \leq \|\tau^{-1}L^{-\nicefrac{\alpha}{2}}\|_{\mathcal{L}(L_2(\Gamma),C^{0,\widetilde{\alpha}}(\Gamma))}^2 d(x,y)^{2\widetilde{\alpha}}.    
\end{equation}
Furthermore, for any $0<\gamma<\widetilde{\alpha}$, $u_0$ has a $\gamma$-H\"older continuous modification. 
\end{Lemma}
\begin{proof}
For $\alpha > \nicefrac12$, we have by Theorem \ref{thm:sobembedding}, that
$
L^{-\nicefrac{\alpha}{2}}:L_2(\Gamma)\to \dot{H}^{\alpha} \cong H^\alpha(\Gamma) \hookrightarrow C^{0,\widetilde{\alpha}}(\Gamma),
$
where $\widetilde{\alpha} = \alpha - \nicefrac12$ if $\alpha \leq 1$ and $\widetilde{\alpha}=\nicefrac12$ if $\alpha>1$.
Therefore, ${L^{-\nicefrac{\alpha}{2}}:L_2(\Gamma)\to C^{0,\widetilde{\alpha}}(\Gamma)}$ is a bounded linear operator. Let $(L^{-\nicefrac{\alpha}{2}})^\ast: \big(C^{0,\widetilde{\alpha}}(\Gamma) \big)^\ast \to \bigl( L_2(\Gamma)\bigr)^\ast = L_2(\Gamma)$
be its adjoint. Define, for $f\in C^{0,\widetilde{\alpha}}(\Gamma)$, ${\<\delta_x, f\> = \int f d \delta_x = f(x)}$. Then $|\<\delta_x,f\>| = |f(x)| \leq \|f\|_{C^{0,\widetilde{\alpha}}(\Gamma)}$, and therefore  $\delta_x \in \bigl(C^{0,\widetilde{\alpha}}(\Gamma) \bigr)^\ast$. Furthermore, $(L^{-\nicefrac{\alpha}{2}})^\ast(\delta_x) \in L_2(\Gamma)$. This tells us that 
$\mathcal{W}\big((L^{-\nicefrac{\alpha}{2}})^\ast(\delta_x)\big)$
is well-defined.

We now use linearity of $\mathcal{W}$, the isometry \eqref{isometryW}, and the considerations above to obtain
\begin{align*}
	\big(\pE\big( | u_0(x) -& u_0(y)|^2 \big)\big)^{\nicefrac12} = \big(\pE\big(\big|\mathcal{W}\big((L^{-\nicefrac{\alpha}{2}}\big)^\ast(\delta_x-\delta_y)\big)\big|^2\big)\big)^{\nicefrac12}
	= \|(L^{-\nicefrac{\alpha}{2}})^\ast(\delta_x-\delta_y)\|_{L_2(\Gamma)} \\
	&\leq \|(L^{-\nicefrac{\alpha}{2}})^\ast\|_{\mathcal{L}\left(\left(C^{0,\widetilde{\alpha}}(\Gamma)\right)^\ast, L_2(\Gamma)\right)} \|\delta_x-\delta_y\|_{\left(C^{0,\widetilde{\alpha}}(\Gamma)\right)^\ast}
	\leq \|L^{-\nicefrac{\alpha}{2}}\|_{\mathcal{L}(L_2(\Gamma),C^{0,\widetilde{\alpha}}(\Gamma))} d(x,y)^{\widetilde{\alpha}},
\end{align*}
where, in the last inequality, we used that, for $f\in C^{0,\widetilde{\alpha}}(\Gamma)$,
$$|\<\delta_x-\delta_y,f\>| = |f(x)-f(y)| \leq \|f\|_{C^{0,\widetilde{\alpha}}(\Gamma)} d(x,y)^{\widetilde{\alpha}},$$
which implies that $\|\delta_x-\delta_y\|_{\left(C^{0,\widetilde{\alpha}}(\Gamma)\right)^\ast}\leq d(x,y)^{\widetilde{\alpha}}$.
This gives us \eqref{holdercontvar}. 

Now, observe that we have that 
$u_0(x)-u_0(y) = \mathcal{W}\bigl((L^{-\nicefrac{\alpha}{2}})^\ast(\delta_x-\delta_y)\bigr)$ 
by the linearity of $\mathcal{W}$.
Therefore, $u_0(x)-u_0(y)$ is a Gaussian random variable with the variance bounded from above by $\|L^{-\nicefrac{\alpha}{2}}\|_{\mathcal{L}(L_2(\Gamma),C^{0,\widetilde{\alpha}}(\Gamma))}^2 d(x,y)^{2\widetilde{\alpha}}$. Thus, for any $p>2$, 
$$\pE\left(|u_0(x) - u_0(y)|^p\right) \leq \pE(|U|^p) \|L^{-\nicefrac{\alpha}{2}}\|_{\mathcal{L}(L_2(\Gamma),C^{0,\widetilde{\alpha}}(\Gamma))}^p d(x,y)^{\widetilde{\alpha}p},$$
where $U$ follows a standard normal distribution. The proof is completed by applying the above version of the Kolmogorov--Chentsov theorem (Theorem \ref{kolmchent}) and standard arguments, since $p>2$ is arbitrary, to obtain the existence of a $\gamma$-H\"older continuous modification for any $0<\gamma<\widetilde{\alpha}$.
\end{proof}

We will now show that for any $0<\gamma<\widetilde{\alpha}$, a $\gamma$-H\"older continuous version of $u_0$ solves \eqref{eq:spde}:

\begin{Lemma}\label{weakregularity}
Fix $\alpha > \nicefrac12$ and let $\widetilde{\alpha} = \alpha-\nicefrac12$ if $\alpha\leq 1$ and $\widetilde{\alpha}=\nicefrac12$ if $\alpha>1$. Fix, also, $0<\gamma<\widetilde{\alpha}$ and let $u$ be any $\gamma$-H\"older continuous modification of $u_0$, where ${u_0(x) = \mathcal{W}\left((\tau^{-1}L^{-\nicefrac{\alpha}{2}})^\ast \delta_x\right)}$. Then, for all $\phi\in L_2(\Gamma)$,
$(u,\phi)_{L_2(\Gamma)} = \mathcal{W}(\tau^{-1}L^{-\nicefrac{\alpha}{2}}\phi).$
\end{Lemma}

\begin{proof}
	Since the set of half-open paths is a semi-ring that generates the Borel sets in $\Gamma$, the simple functions obtained by the above indicator functions are dense in $L_2(\Gamma)$, and thus we first show the results for simple function, and then for any function in $L_2(\Gamma)$. 
Recall that given $x,y\in\Gamma$,  $[x,y]\subset \Gamma$ is the path connecting $x$ and $y$ with shortest length, and let $(x,y] = [x,y]\setminus\{x\}$.
We will begin by considering $x,y\in e$, for some $e\in\mathcal{E}$ and showing that
$\int_{(x,y]} u(s) ds = \mathcal{W}(L^{-\nicefrac{\alpha}{2}}1_{(x,y]}).$
Observe that
\begin{equation}
	\int_{(x,y]} u(s) ds = \int_{[x,y]} u(s)ds.
\end{equation}
Therefore, we can take the closed path $[x,y]$. Since $u$ has continuous sample paths, the integral
$ \int_{[x,y]} u(s)ds$ is the limit of Riemann sums. Specifically, let $x=x_0,\ldots,x_n=y$ be a partition of $[x,y]$, and let $x_i^\ast$ be any point in $[x_i,x_{i+1}]$, $i=0,\ldots,n-1$, where $\max_i l([x_i,x_{i+1}])\to 0$ as $n\to\infty$. Then, 
\begin{align*}
	\int_{[x,y]} u(s)ds &= \lim_{n\to\infty} \sum_{i=0}^{n-1} u(x_i^\ast) l([x_i,x_{i+1}])
	= \lim_{n\to\infty} \sum_{i=0}^{n-1} \mathcal{W}\left((L^{-\nicefrac{\alpha}{2}})^\ast \delta_{x_i^\ast}\right) l([x_i,x_{i+1}])\\
	&= \lim_{n\to\infty} \mathcal{W}\Big((L^{-\nicefrac{\alpha}{2}})^\ast\Big(\sum_{i=0}^{n-1} \delta_{x_i^\ast} l([x_i,x_{i+1}]) \Big) \Big),
\end{align*}
by the linearity of $\mathcal{W}$.
Observe that $C^{0,\widetilde{\alpha}}(\Gamma)\hookrightarrow C^{0,\gamma}(\Gamma)$, since $[f]_{C^{0,\gamma}(\Gamma)} \leq \hbox{diam}(\Gamma)^{\widetilde{\alpha}-\gamma}[f]_{C^{0,\widetilde{\alpha}}(\Gamma)}$. 
Let us now study the convergence of $\sum_{i=0}^{n-1} \delta_{x_i^\ast} l([x_i,x_{i+1}])$ in $\big(C^{0,\gamma}(\Gamma)\bigr)^\ast$. We have, for any ${f\in C(\Gamma)}$
\begin{align*}
	\Bigl| \<1_{[x,y]},f\> - & \Big\< \sum_{i=0}^{n-1} \delta_{x_i^\ast} l([x_i,x_{i+1}]) ,f\Big\> \Bigr| \leq \sum_{i=0}^{n-1} \int_{[x_i,x_{i+1}]} |f(s) - f(x_i^\ast)| ds\\
	&\leq \|f\|_{C^{0,\gamma}(\Gamma)} \sum_{i=0}^{n-1} \int_{[x_i,x_{i+1}]} d(x_i,x_{i+1})^\gamma ds
	\leq \|f\|_{C^{0,\gamma}(\Gamma)} d(x,y) \max_{1\leq i\leq n} d(x_i,x_{i+1})^\gamma. 
\end{align*}
Therefore, since $\max_{1\leq i\leq n} l([x_i,x_{i+1}])\to 0$ as $n\to\infty$, it follows that
$$
\lim_{n \rightarrow \infty}	\Bigl\|1_{[x,y]} - \sum_{i=0}^{n-1} \delta_{x_i^\ast} l([x_i,x_{i+1}]) \Bigr\|_{\left(C^{0,\gamma}(\Gamma)\right)^\ast} \leq \lim_{n \rightarrow \infty}	d(x,y) \max_{1\leq i\leq n} d(x_i,x_{i+1})^\gamma \to 0.
$$
Furthermore, since 
$(L^{-\nicefrac{\alpha}{2}})^\ast: (C^{0,\gamma}(\Gamma))^\ast\to L_2(\Gamma)$ is a bounded operator, we have that, as   $n \rightarrow\infty$, 
$
(L^{-\nicefrac{\alpha}{2}})^\ast \Big( \sum_{i=0}^{n-1} \delta_{x_i^\ast} l([x_i,x_{i+1}])\Big) \to (L^{-\nicefrac{\alpha}{2}})^\ast 1_{[x,y]}
$
in $L_2(\Gamma)$. Finally, by \eqref{isometryW}, we have that
$$
\lim_{n\to\infty} \mathcal{W}\Big((L^{-\nicefrac{\alpha}{2}})^\ast\Big(\sum_{i=0}^{n-1} \delta_{x_i^\ast} l([x_i,x_{i+1}]) \Big) \Big) = \mathcal{W}((L^{-\nicefrac{\alpha}{2}})^\ast 1_{[x,y]}) = \mathcal{W}(L^{-\nicefrac{\alpha}{2}}1_{[x,y]}).
$$
	This proves that
	$\int_{[x,y]} u(s) ds = \mathcal{W}(L^{-\nicefrac{\alpha}{2}}1_{[x,y]}).$ Hence, by linearity, the result follows for simple functions.

We now move to the general case.
Let $\phi\in L_2(\Gamma)$ be any function and $\phi_n$ be a sequence of simple functions such that $\phi_n\to \phi$ in $L_2(\Gamma)$. Then,
\begin{equation*}
	\pE\Bigl(\Bigl| (u,\phi)_{L_2(\Gamma)}  - \mathcal{W}(L^{-\nicefrac{\alpha}{2}}\phi) \Bigr|^2\Bigr)^{\nicefrac12} \leq \pE\Bigl(\Bigl|(u,\phi)_{L_2(\Gamma)} - (u,\phi_n)_{L_2(\Gamma)} \Bigr|^2\Bigr)^{\nicefrac12}
	+ \pE\Bigl(\Bigl|\mathcal{W}(L^{-\nicefrac{\alpha}{2}}(\phi_n - \phi)) \Bigr|^2\Bigr)^{\nicefrac12}.
\end{equation*}
Observe that
$\pE\big(\big|(u,\phi)_{L_2(\Gamma)} - (u,\phi_n)_{L_2(\Gamma)} \big|^2\big)^{\nicefrac12}\to 0$
as $n\to\infty$. Indeed, since $u$ is a modification of $u_0$, we have that, for every
$x\in \Gamma$,
\begin{equation}\label{eq:expsquareineq}
	\begin{aligned}
		\pE\left(|u(x)|^2 \right)^{\nicefrac12} &= \pE\left(|u_0(x)|^2\right)^{\nicefrac12} = \left\|(L^{-\nicefrac{\alpha}{2}})^\ast(\delta_x) \right\|_{L_2(\Gamma)}
		\leq \left\|(L^{-\nicefrac{\alpha}{2}})^\ast \right\|_{\mathcal{L}(C^{0,\widetilde{\alpha}}(\Gamma)^\ast,L_2(\Gamma))} \|\delta_x \|_{C^{0,\widetilde{\alpha}}(\Gamma)^\ast}\\
		& = \left\|(L^{-\nicefrac{\alpha}{2}})^\ast \right\|_{\mathcal{L}(L_2(\Gamma),C^{0,\widetilde{\alpha}}(\Gamma))} \| \delta_x\|_{C^{0,\widetilde{\alpha}}(\Gamma)^\ast}.
	\end{aligned}
\end{equation}
Finally since $|\langle \delta_x , f\rangle | = |f(x)| \leq \|f\|_{C^{0,\widetilde{\alpha}}(\Gamma)}$ it follows that
$\|\delta_x\|_{C^{0,\widetilde{\alpha}}(\Gamma)^\ast} \leq 1$, and hence 
\begin{align}
	\left\|(L^{-\nicefrac{\alpha}{2}})^\ast \right\|_{\mathcal{L}(L_2(\Gamma),C^{0,\widetilde{\alpha}}(\Gamma))} \| \delta_x\|_{C^{0,\widetilde{\alpha}}(\Gamma)^\ast} \leq\left\|(L^{-\nicefrac{\alpha}{2}})^\ast \right\|_{\mathcal{L}(L_2(\Gamma),C^{0,\widetilde{\alpha}}(\Gamma))}.\label{eq:ineqvarmod} 
\end{align}
Therefore,
by \eqref{eq:expsquareineq}, \eqref{eq:ineqvarmod} and Fubini's theorem, we have that
\begin{align}
	\pE\left[\|u\|_{L_2(\Gamma)}^2 \right] = \pE\Bigl[\sum_{e\in\mathcal{E}} \int_0^{\ell_e} u_e^2(t)dt \Bigr]
	= \sum_{e\in\mathcal{E}} \int_{0}^{\ell_e} \pE(u_e^2(t)) dt \leq |\Gamma| \|L^{-\nicefrac{\alpha}{2}}\|_{\mathcal{L}(L_2(\Gamma), C^{0,\widetilde{\alpha}}(\Gamma))}^2.\label{eq:ineqvarmod_part2}
\end{align}	
Now, \eqref{eq:ineqvarmod_part2} and the fact that $\phi_n\to \phi$ in $L_2(\Gamma)$, 
imply 
{$\pE\big(\big|(u,\phi)_{L_2(\Gamma)} - (u,\phi_n)_{L_2(\Gamma)} \big|^2\big)^{\nicefrac12}\to 0$}
.
On the other hand, since $L^{-\nicefrac{\alpha}{2}}:L_2(\Gamma) \to L_2(\Gamma)$ is bounded, we have that $L^{-\nicefrac{\alpha}{2}}(\phi_n - \phi)\to 0$ in $L_2(\Gamma)$ and by \eqref{isometryW}, we obtain
$\pE\big(\big|\mathcal{W}(L^{-\nicefrac{\alpha}{2}}(\phi_n - \phi)) \big|^2\big)^{\nicefrac12}\to 0$
as $n\to\infty$.
Thus, $(u,\phi)_{L_2(\Gamma)} = \mathcal{W}(L^{-\nicefrac{\alpha}{2}}\phi)$ follows.
\end{proof}

\begin{proof}[Proof of Theorem \ref{thm:weakregularity}]
Lemma \ref{kolmchentbounds} and Lemma \ref{weakregularity} directly imply that $u$ is a solution of $L^{\nicefrac{\alpha}{2}} u = \mathcal{W}$. Therefore, for any $0<\gamma<\nicefrac{1}{2}$, $u$ has a modification with $\gamma$-H\"older continuous sample paths.
Now, if $\alpha = 3$, then
$L^{\nicefrac32} u = \mathcal{W}$ if and only if $v = Lu = L^{-\nicefrac12}\mathcal{W}.$
This implies that $v$ solves $L^{\nicefrac12}v = \mathcal{W}$ and by Lemma~\ref{weakregularity}, $v$ has $\gamma$-H\"older continuous sample paths. 
Finally, note that
$v = Lu = \kappa^2 u - u''.$
By using the first part of this proof directly to $u$, we also have that $u$ has a modification with $\gamma$-H\"older continuous sample paths. Therefore, also $D^2 u = u'' = v + \kappa^2 u$  has a modification with $\gamma$-H\"older continuous sample paths. The general case can be handled similarly.
\end{proof}

\section{Spectral representations and simulation}\label{sec:spectralrepr}
In this section, we will further characterize the covariance function of the Whittle--Mat\'ern fields and show how they can be simulated through Karhunen--Lo\`eve expansions. 
We directly have:

\begin{Corollary}\label{cor:covfunccont}
Fix $\alpha > \nicefrac{1}{2}$, then the covariance function $\varrho(\cdot,\cdot)$ in \eqref{eq:covfunc} is  continuous.
\end{Corollary}

\begin{proof}
Observe that since $u$ is a modification of $u_0$, we have, by Lemma \ref{kolmchentbounds}, that 
\begin{equation}\label{contcovfunc}
	\pE\left(|u(s) - u(s')|^2\right) = \pE\left(|u_0(s) - u_0(s')|^2\right) \leq \|\tau^{-1}L^{-\nicefrac{\alpha}{2}}\|_{\mathcal{L}(L_2(\Gamma),C^{0,\widetilde{\alpha}}(\Gamma))}^2 d(s,s')^{2\widetilde{\alpha}}.
\end{equation}
This implies that $u$ is $L_2$ continuous at each $s\in\Gamma$, so $\varrho$ is continuous at each ${(s,s')\in\Gamma\times\Gamma}$.
\end{proof}

We will now show that the covariance operator is an integral operator with kernel given by the covariance function. 

\begin{Lemma}\label{prp:covintegralrep}
Let $\alpha > \nicefrac12$ and $\mathcal{C}= \tau^{-2} L^{-\alpha}$ be the covariance operator of $u$, where $u$ is the solution of $L^{\nicefrac{\alpha}{2}} (\tau u) = \mathcal{W}$. Then,
$$(\mathcal{C}\phi)(s') =  (\varrho(\cdot, s'), \phi)_{L_2(\Gamma)} =  \sum_{e\in\mathcal{E}} \int_e \varrho(s,s')\phi(s)ds,\quad s'\in\Gamma,$$
where $\varrho$ is the covariance function of $u$.
\end{Lemma}

\begin{proof}
Assume without loss of generality that $\tau=1$ and 
let $T:L_2(\Gamma)\to L_2(\Gamma)$ be the integral operator with kernel $\varrho$,
$(T\phi)(s') = (\varrho(\cdot, s'), \phi)_{L_2(\Gamma)} =  \sum_{e\in\mathcal{E}} \int_e \varrho(s,s')\phi(s)ds.$
Now, fix $\alpha > \nicefrac12$ and recall that, by Corollary \ref{cor:covfunccont}, $\varrho$ is continuous. Since, $\Gamma$ is compact, $\varrho$ is bounded, say by $K>0$. Further, since $\Gamma$ has finite measure, $L_2(\Gamma)\subset L_1(\Gamma)$, so for any $\phi,\psi\in L_2(\Gamma)$, we have, by the Cauchy-Schwarz inequality,
\begin{align*}
	\sum_{e,\widetilde{e}\in\mathcal{E}}  \int_{\widetilde{e}}\int_{e} \pE(|u(s)u(s') \phi(s)\psi(s')|) ds ds' &\leq \sum_{e,\widetilde{e}\in\mathcal{E}}  \int_{\widetilde{e}}\int_{e} \sqrt{\varrho(s,s)\varrho(s',s')} |\phi(s)| |\psi(s')| dsds'\\
	&\leq K \|\phi\|_{L_1(\Gamma)} \|\psi\|_{L_1(\Gamma)} <\infty.
\end{align*}
Therefore, by Fubini's theorem,
\begin{align*}
	(T\phi, \psi)_{L_2(\Gamma)} &=  \sum_{e,\widetilde{e}\in\mathcal{E}}  \int_{\widetilde{e}} \int_{e} \varrho(s,s') \phi(s) ds\psi(s') ds'
	= \sum_{e,\widetilde{e}\in\mathcal{E}}   \int_{\widetilde{e}} \int_{e} \pE(u(s)u(s')) \phi(s) ds\psi(s') ds' \\
	&= \sum_{e,\widetilde{e}\in\mathcal{E}}  \pE\left( \int_{\widetilde{e}} \left( \int_e u(s) \phi(s) ds\right) u(s') \psi(s') ds' \right)\\
	& = \sum_{e,\widetilde{e}\in\mathcal{E}}  \pE\left((u,\phi)_{L_2(e)}(u,\psi)_{L_2(e)}\right) = \pE\left((u,\phi)_{L_2(\Gamma)} (u,\psi)_{L_2(\Gamma)} \right)
	= (\mathcal{C} \phi,\psi)_{L_2(\Gamma)}.
\end{align*}

This shows that the covariance operator is the integral operator with kernel $\varrho(\cdot,\cdot)$.
\end{proof}

Since the covariance operator is an integral operator, we have the  following series expansion of the covariance function. 
\begin{Proposition}\label{prp:mercercov}
Let $\alpha > \nicefrac12$, then the covariance function $\varrho$ admits the series representation  in terms of the eigenvalues and eigenvectors of $L$:
$$
\varrho(s,s') = \sum_{i=1}^{\infty} \frac{\tau^{-2}}{(\kappa^2 + \hat{\lambda}_i)^{\alpha}} \eig_i(s)\eig_i(s'),
$$
where the convergence of the series is absolute and uniform. The series also converges in $L_2(\Gamma\times\Gamma)$.
\end{Proposition}

\begin{proof}
It is a direct consequence of Lemma \ref{prp:covintegralrep}, Corollary \ref{cor:covfunccont} and Mercer's theorem \cite{Steinwart2012}. Furthermore, since the $\Gamma$ has finite measure, it follows by the dominated convergence theorem that the series also converges in $L_2(\Gamma\times\Gamma)$.
\end{proof}

The continuity of the covariance function implies that we can represent $u$ through the Karhunen--Lo\`eve expansion as explained in the following proposition.

\begin{Proposition}\label{prp:KLexp}
Let $\alpha > \nicefrac{1}{2}$, and let $u$ be the solution of \eqref{eq:spde}. Then, 
$
u(s) = \tau^{-1}\sum_{i=1}^{\infty} \xi_i (\kappa^2 + \hat{\lambda}_i)^{-\nicefrac{\alpha}{2}} \eig_i(s),
$
where $\xi_i$ are independent standard Gaussian variables. Further, the series 
\begin{equation}\label{eq:un}
	u_n(s) = \tau^{-1}\sum_{i=1}^n \xi_i (\kappa^2 + \hat{\lambda}_i)^{-\nicefrac{\alpha}{2}} \eig_i(s),
\end{equation} converges in $L_2(\Omega)$ uniformly in $s$, that is,
$\lim_{n\to\infty} \sup_{s\in \Gamma} \pE\left(\left| u(s) - u_n(s)\right|^2 \right) =0.$
\end{Proposition}

\begin{proof}
Assume, without loss of generality, that $\tau=1$  and 
define for each $i\in\mathbb{N}$ the random variable ${\xi_i = \lambda_i^{\nicefrac{\alpha}{2}} (u, \eig_i)_{L_2(\Gamma)}}$.  By Lemma \ref{weakregularity}, we have that 
$\xi_i = \lambda_i^{\nicefrac{\alpha}{2}} \mathcal{W}(L^{-\nicefrac{\alpha}{2}}\eig_i) =  \mathcal{W}(\eig_i)$, and therefore, $\xi_i$ are centered Gaussian random variables. 
Further, by \eqref{isometryW}, we have $\pE(\xi_i^2) = 1$ and if $i\neq j$, $\pE(\xi_i \xi_j) = (\eig_i,\eig_j)_{L_2(\Gamma)} = 0.$ 
Hence, by the Gaussianity of $\xi_i$, $\{\xi_i\}_{i\in\mathbb{N}}$ is a sequence of independent standard Gaussian random variables. 
Thus, for every $s\in\Gamma$,
\begin{align*}
	\pE\left( |u(s) - u_n(s)|^2\right) &=  \pE(u_n(s)^2) - 2 \pE(u_n(s) u(s)) + \pE(u(s)^2)\\
	&= \sum_{i=1}^n \frac1{\lambda_i^{\alpha}} \eig_i(s)^2 - 2\sum_{i=1}^n \frac1{\lambda_i^{\nicefrac{\alpha}{2}}}\pE(u(s)\xi_i) \eig_i(s) + \varrho(s,s).
\end{align*}
Now, by Fubini's theorem (with a similar justification as the one in the proof of Lemma~\ref{prp:covintegralrep}) and the fact that $\rho$ is the kernel of the covariance operator $L^{-\alpha}$, we have that
\begin{align*}
	\pE(u(s)\xi_i) &= (\lambda_i^{\nicefrac{\alpha}{2}}\pE(u(s) u), \eig_i)_{L_2(\Gamma)} = \lambda_i^{\nicefrac{\alpha}{2}}(\varrho(s,\cdot),\eig_i)_{L_2(\Gamma)}
	= \lambda_i^{\nicefrac{\alpha}{2}}(L^{-\alpha} \eig_i)(s) = \lambda_i^{-\nicefrac{\alpha}{2}} \eig_i(s).
\end{align*}
Thus, by the previous computations and Corollary \ref{prp:mercercov}, we have that
$$ 
\lim_{n\to\infty} \sup_{s\in\Gamma} \pE\left( |u(s) - u_n(s)|^2\right) =  \lim_{n\to\infty} \sup_{s\in\Gamma} \Bigl( \varrho(s,s) - \sum_{i=1}^n \frac1{(\kappa^2 + \hat{\lambda}_i)^{\alpha}} \eig_i(s)^2 \Bigr) = 0.
$$
\end{proof}

Given that we are on a graph where the eigenfunctions of $\Delta_{\Gamma}$ are known, we can use the Karhunen--Lo\`eve expansion as a simulation method by using the truncated series expansion \eqref{eq:un}. 

\begin{example}\label{ex:tadpole}
As a concrete example, consider the graph in Figure~\ref{fig:cov_example}. 
Since we can remove vertices of degree two without changing the model 
(see Proposition~\ref{prop:join}), and since the edges are parameterized in terms of arc length, the graph  may be replaced by a so-called tadpole graph consisting of two vertices and two edges as shown in Figure~\ref{fig:tadpole}. 
Tadpole graphs are well-studied in the quantum graph literature and are important building blocks for more sophisticated graphs \cite[see, e.g.][]{Serio2021}. We assume that the left edge $e_1$ has length $1$ and that the circular edge $e_2$ has length $2$, and parameterize a point on $e_1$ as $s = (e_1,t)$ for $t\in [0,1]$ and a point on $e_2$ by $s = (e_2,t)$ for $t\in[0,2]$
Then it is easy to verify that the constant function $\phi_0(s) = 1/\sqrt{3}$ and the two sets $\{\phi_i\}_{i\in\bbN}$ and $\{\psi_i\}_{2i\in\bbN}$ form the eigenfunctions of $-\Delta_\Gamma$ with eigenvalues $0$, $\{(i\pi/2)^2\}_{i\in\bbN}$ and  $\{(i\pi/2)^2\}_{2i\in\bbN}$, respectively. Here
$$
\phi_i(s) = C_{\phi,i} \begin{cases}
	-2\sin\left(\nicefrac{i\pi}{2}\right)\cos\left(\nicefrac{i\pi t}{2}\right) & s \in e_1,\\
	\sin\left(\nicefrac{i\pi t}{2}\right) & s \in e_2,
\end{cases},
\qquad 
\psi_i(s) = \frac{\sqrt{3}}{\sqrt{2}} \begin{cases}
	(-1)^{\nicefrac{i}{2}}\cos\left(\nicefrac{i\pi t}{2}\right) & s \in e_1,\\
	\cos\left(\nicefrac{i\pi t}{2}\right) & s \in e_2,
\end{cases}
$$
with $C_{\phi,i} = 1$ if $i$ is even and $C_{\phi,i} = 1/\sqrt{3}$ otherwise, and these functions form an orthonormal basis for $L_2(\Gamma)$. 
Thus, we can use Proposition~\ref{prp:mercercov} to evaluate the covariance function and Proposition~\ref{prp:KLexp} to simulate the field. 
This is how Figure~\ref{fig:cov_example}  was created, and an example with $\alpha=2.5$ is shown in Figure \ref{fig:tadpole}.
\end{example}

The rate of convergence  of the approximation \eqref{eq:un} is clarified in the following proposition. 

\begin{Proposition}\label{prp:rateKL}
If $\alpha > \nicefrac{1}{2}$, $u_n$ converges to $u$ in  $L_2(\Omega, L_2(\Gamma))$,
${\lim_{n\to\infty}  \left\| u - u_n\right\|_{L_2(\Omega,L_2(\Gamma))} = 0}$,
and there exists some constant $C>0$ such that for all $n\in\mathbb{N}$,
$
\|u-u_n\|_{L_2(\Omega, L_2(\Gamma))} \leq C n^{-(\alpha-\nicefrac12)}.
$
\end{Proposition}
\begin{proof}
Assume without loss of generality that $\tau=1$.
Note that by Proposition \ref{prp:KLexp} combined with the fact that $\Gamma$ has finite measure, it follows from the dominated convergence theorem $u_n$ converges to $u$ in the $L_2(\Omega,L_2(\Gamma))$ norm. To simplify the notation, we write $L_2 = L_2(\Omega,L_2(\Gamma))$.
Let $N>n$ and observe that
\begin{align*}
	\left\| u_N - u_n\right\|_{L_2}^2 &= \Bigl\|  \sum_{i=n+1}^{N} \xi_i (\kappa^2 + \hat{\lambda}_i)^{-\nicefrac{\alpha}{2}} \eig_i \Bigr\|_{L_2}^2\\
	&=  \sum_{e\in\mathcal{E}} \int_e \pE\left( \sum_{i=n+1}^N \sum_{j=n+1}^N (\kappa^2 + \hat{\lambda}_i)^{-\nicefrac{\alpha}{2}}(\kappa^2 + \hat{\lambda}_j)^{-\nicefrac{\alpha}{2}} \xi_i \xi_j \eig_i(s) \eig_j(s)\right) ds  \\
	&=  \sum_{i=n+1}^N  (\kappa^2 + \hat{\lambda}_i)^{-\alpha} \pE(\xi_i^2) \sum_{e\in\mathcal{E}} \int_e \eig_i(s)^2ds =  \sum_{i=n+1}^N  (\kappa^2 + \hat{\lambda}_i)^{-\alpha}.
\end{align*}	
Since we have convergence of $u_N$ to $u$ in the $L_2$-norm, we obtain by the above expression, the Weyl asymptotics \eqref{eq:weyl} and the integral test for series, that
\begin{align*}
	\left\| u - u_n\right\|_{L_2}^2 &= \lim_{N\to\infty} \left\| u_N - u_n\right\|_{L_2}^2
	= \lim_{N\to\infty} \left\|  \sum_{i=n+1}^{N} \xi_i (\kappa^2 + \hat{\lambda}_i)^{-\nicefrac{\alpha}{2}} \eig_i \right\|_{L_2}^2\\
	&= \lim_{N\to\infty} \sum_{i=n+1}^N (\kappa^2 + \hat{\lambda}_i)^{-\alpha} = \sum_{i=n+1}^\infty (\kappa^2 + \hat{\lambda}_i)^{-\alpha} 
	\leq C  \sum_{i=n+1}^\infty i^{-2\alpha} \leq \frac{C}{n^{2\alpha-1}}.
\end{align*}
\end{proof}

We end this section by verifying the rate of the spectral simulation method derived in Proposition~\ref{prp:rateKL} numerically for the previous example.

\begin{figure}[t] 
\begin{center}
	\includegraphics[width=0.45\linewidth]{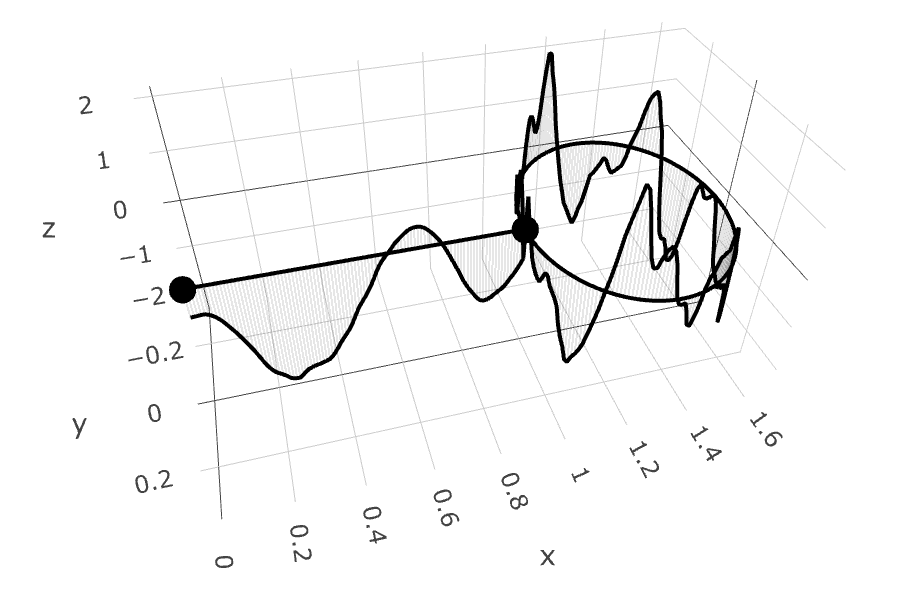}
	\includegraphics[width=0.4\linewidth]{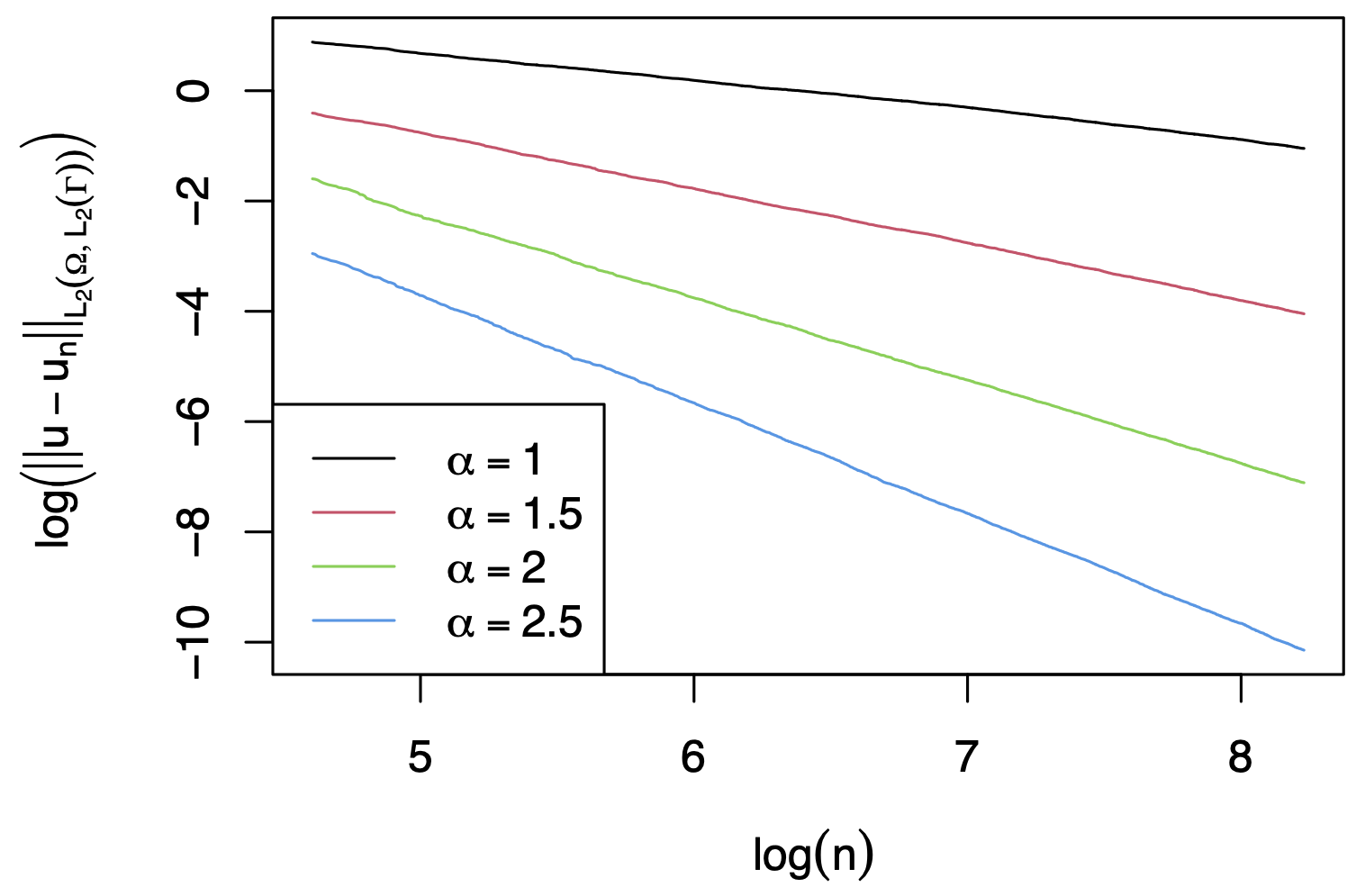}
\end{center}
\vspace{-0.5cm} 
\caption{Example of a Whittle--Mat\'ern field with $\alpha=2.5$ and $\kappa=10$ on a tadpole graph (left) and log-log plot of errors of the spectral simulation method (right).  }
\label{fig:tadpole}
\end{figure}
\begin{example}
Consider the tadpole graph in Example~\ref{ex:tadpole}. We  construct reference simulations based on overkill Karhunen--Lo\` eve expansions. More precisely,
we fix the parameters $(\kappa,\tau,\alpha)$ and use \eqref{eq:un} with $n = n_{ok} = 15000$  to generate $N$ independent reference simulations 
$u_{ok}^{(i)}$ for $i=1,\ldots, N$. To approximate the $L_2(\Omega, L_2(\Gamma))$ errors for a simulation with a fixed $n<n_{ok}$, we compute
$$
\|u_{n_{ok}} - u_n\|_{L_2(\Omega, L_2(\Gamma))}^2 \approx \frac1{N m} \sum_{i=1}^N \sum_{j=1}^m \left(u_{n_{ok}}^{(i)}(s_j) - u_n^{(i)}(s_j)\right)^2,
$$
where $s_1,\ldots, s_m$ are evenly spaced locations over the  graph. We in particular used $m=300$ and ${N=10}$ and considered four different values of $\alpha$, $\{1,1.5,2,2.5\}$. The other  parameters were fixed 
as $\kappa=10$ and $\tau^2 = \Gamma(\alpha-\nicefrac12)/(\Gamma(\alpha)\sqrt{4\pi}\kappa^{2\alpha-1})$. Here, the choice of $\tau$ makes the variance of the process approximately equal to 1 at locations away from the vertices. The errors for four different sets of parameters are shown in Figure \ref{fig:tadpole} for $n=100,\ldots, 3750$. As can be seen in that log-log plot, the errors are decreasing linearly, and by fitting a line to the data for each case using ordinary least squares we obtain the estimated convergence rates shown in Table \ref{tab:rates}, which validate the theoretical rates.

\begin{table}[t]
	\caption{Observed (theoretical) rates of convergence for the strong errors  shown in Figure~\ref{fig:tadpole}.}
	{\centering
		\begin{tabular}{rcccc}
			\toprule
			$\alpha$ &  $1$  & $1.5$ & $2$ & $2.5$ \\
			\cmidrule(r){2-5}
			Convergence rates & 0.558 (0.5) & 1.013 (1.0)& 1.504 (1.5) & 1.989 (2.0)\\
			\bottomrule
	\end{tabular}}
	\label{tab:rates}
\end{table}
\end{example}

\section{Mean-square differentiability}\label{sec:mean-square}
Because of the continuity of the covariance function (Corollary \ref{cor:covfunccont}), we have that for every ${\alpha>\nicefrac12}$, the solution of $u$ of 
\eqref{eq:spde} is $L_2(\Omega)$ continuous. However, for $\alpha\geq 2$ we can say more. To such an end, we need some additional
definitions.

Let $H_u = \overline{\textrm{span}\{u(s): s\in\Gamma\}}$. Let $e$ be an edge and let $f:e \to H_u$ be a function. 
We say that $f$ is weakly differentiable at $s$ in the $L_2(\Omega)$ sense if there exists $f'(s)\in H_u$
such that for every $w\in H_u$ and every sequence $s_n\to s$, with $s_n\neq s$, we have $\pE(w(f(s_n)-f(s))/(s_n-s)) \to \pE(wf'(s))$. 
We define the higher order weak derivatives in the $L_2(\Omega)$ sense inductively: for $k\geq 2$ we say that $f$ has $k$th order weak
derivative at $s\in e$ if $f^{(k-1)}(s')$ exists for every $s'\in e$ and it is weakly differentiable at $s$. Finally, we say that a function
$f: e \to H_u$ is weakly continuous in the $L_2(\Omega)$ sense if for every $w\in H_u$, the function $s\mapsto \pE(w f(s))$ is continuous.

The following Lemma studies the existence (and weak continuity) of weak derivatives in the $L_2(\Omega)$ sense for solutions of \eqref{eq:spde}: 

\begin{Lemma}\label{prp:oddorderderiv}
Let $\alpha\geq 2$ and let $u$ be a solution of $L^{\nicefrac{\alpha}{2}} (\tau u) = \mathcal{W}$. Then, for each edge $e\in\mathcal{E}$, 
the directional derivatives of $u$ up to order $\lfloor\alpha\rfloor-1$ exist weakly in the $L_2(\Omega)$ sense and are weakly continuous on $e$ in the $L_2(\Omega)$ 
sense (on the boundary of the edges one must consider the lateral derivatives). 
\end{Lemma}

\begin{proof}
We begin by briefly introducing (and characterizing) the Cameron-Martin space (also known as reproducing kernel Hilbert space) associated to $u$. 
Let $\mathcal{H}_u = \{h(s) = \pE(w u(s)): s\in\Gamma, w\in H_u\}$. Then, $\mathcal{H}_u$ is the Cameron-Martin 
space associated to $u$. Now, it follows from 
Proposition \ref{prp:KLexp} that
$\mathcal{H}_u = \dot{H}^\alpha.$
Take any $h \in \mathcal{H}_u$, so that $h\in\dot{H}^\alpha$ and note that for each edge $e\in\mathcal{E}$, we have, from Proposition \ref{prp:hdotk}, that
$$\|h\vert_e\|_{H^{\lfloor\alpha\rfloor}(e)} \leq \|h\|_{\widetilde{H}^{\lfloor\alpha\rfloor}(\Gamma)} \leq C \|h\|_{\dot{H}^{\lfloor\alpha\rfloor}},$$
for some constant $C>0$. This tells us that for each edge $e\in\mathcal{E}$, $h\vert_e \in H^{\lfloor\alpha\rfloor}(e)$, and it follows from standard Sobolev embedding (or directly by noticing
that $H^1(e)$ coincides with the space of absolutely continuous functions), that for each $h\in\mathcal{H}_u$ and each $e\in\mathcal{E}$, $h\vert_e \in C^{\alpha-1}(e)$.

Fix an edge $e\in\mathcal{E}$. We have that for any $w\in H_u$, the map $h(t) = \pE(w u_e(t))$ is continuously differentiable in $[0,\ell_e]$. Now, note that this gives
us that for any $t\in [0,\ell_e]$ and any sequence $t_n \to t$, $t_n \in [0,\ell_e]$ and $t_n\neq t$, the sequence $(u_e(t_n)-u_e(t))/(t_n-t)$ is weakly Cauchy in $H_u$. Since $H_u$ is 
Hilbert space, it is weakly sequentially complete. Thus, there exists some $u_e'(t)\in H_u$ such that 
\begin{equation}\label{eq:weakconvderiv}
	\frac{u_e(t_n)-u_e(t)}{t_n-t} \stackrel{w}{\longrightarrow} u_e'(t).
\end{equation}
Since $h(t) = \pE(w u_e(t))$ is differentiable, the limit $u_e'(t)$ does not depend on the choice of the sequence $t_n\to t$. Furthermore, the weak convergence also 
gives us that $h'(t) = \pE(w u_e'(t))$ (where we take lateral derivatives if $t$ belongs to the boundary of $[0,\ell_e]$). 
This proves the existence of the weak derivative in the $L_2(\Omega)$-sense. Furthermore, since $h$ is continuously differentiable, we have that $h'(t)$ is continuous,
so $u_e'$ is weakly continuous in the $L_2(\Omega)$ sense.

Let us now consider the case ${\lfloor\alpha\rfloor}>2$. For each $t\in [0,\ell_e]$, $\widetilde{h}(t,\cdot)\in H^{\alpha-1}([0,\ell_e]) \subset C^1([0,\ell_e])$. 
Similarly, for each $t\in [0,\ell_e]$, $\widetilde{h}(\cdot, t)\in C^1(e)$. This gives us that $\widetilde{h} \in C^1([0,\ell_e]\times [0,\ell_e])$, so $\widetilde{h}$ is, in particular,
continuous in $[0,\ell_e]\times [0,\ell_e]$, which gives us that $u_e'$ is the $L_2(\Omega)$-derivative of $u$ in $e$ and it is $L_2(\Omega)$ continuous. 
Thus, we can iterate this argument until we obtain that $u_e', \ldots, u_e^{({\lfloor\alpha\rfloor}-2)}$ are the $L_2(\Omega)$-derivatives and are continuous in 
the $L_2(\Omega)$ sense. Finally, the case $u_e^{({\lfloor\alpha\rfloor}-1)}$ is handled in the same fashion as we handled $u_e'(\cdot)$ in the ${\lfloor\alpha\rfloor}=2$ case.
\end{proof}

\begin{Remark}
Observe that we have, from the proof of Lemma \ref{prp:oddorderderiv}, that for $\alpha\geq 2$, and for each edge $e$, the derivatives 
$u_e',\ldots, u_e^{(\lfloor\alpha\rfloor -2)}$ exist in the strong sense, that is, these derivatives are the $L_2(\Omega)$ derivatives,
and they are (strongly) continuous in the $L_2(\Omega)$ sense.
\end{Remark}

We will now show that the weak derivatives (in the Sobolev sense) and the weak derivatives in the $L_2(\Omega)$ sense of the solutions
of \eqref{eq:spde} coincide for almost every point in $\Gamma$. Since the weak derivatives in the $L_2(\Omega)$ sense are weakly continuous (in the $L_2(\Omega)$ sense),
the next result tells us that the point evaluation of the weak derivative (in the Sobolev sense) is well-defined (by choosing the weak derivative in $L_2(\Omega)$
sense as a modification of the weak derivative in the Sobolev sense). Furthermore, since they agree, there is no confusion in using the same notation for both
the weak derivative in the Sobolev sense and the weak derivative in the $L_2(\Omega)$ sense.

\begin{Proposition}\label{cor:L2diffweakdiff}
If $\alpha\geq 2$, then the weak derivative of solution $u$ of \eqref{eq:spde} in the 
$L_2(\Omega)$ sense coincide at almost every point of $\Gamma$ with the weak derivative of $u$ in the Sobolev sense 
(if $s$ is a boundary term, then one must consider a lateral derivative). Moreover, the same result is true for higher order derivatives when they exist.
\end{Proposition}

\begin{proof}
Let $e\in\mathcal{E}$ be an edge. Throughout the proof, we will denote the weak derivative of $u_e$ in the Sobolev sense by $u_e'$ and the
weak derivative in the $L_2(\Omega)$ sense by $\widetilde{u}_e'$. For a real function $f$, let $\Delta_h f(t) = (f(t+h) - f(t))/h$ be the difference quotient. It is well-known (see, for instance, \cite[Theorem 3, p.292]{evans2010partial}) 
that for any closed set $I$ contained in the interior of $[0,\ell_e]$ (by allowing lateral derivative, in which case we would consider $h>0$ or $h<0$, 
we can assume that $I$ contains one of the boundary points of $e$), and any $h\in\mathbb{R}$ such that $0< |h| < dist(I, \partial e)$, we have
\begin{equation}\label{eq:quotdiffderivineq}
	\|\Delta_h u_e\|_{L_2(I)} \leq C \|u_e'\|_{L_2(e)},
\end{equation}
for some constant $C>0$. By this inequality, the fact that $u_e\in H^1([0,\ell_e])$, and a simple argument by approximation by smooth functions, one can readily prove the well-known result that
$\Delta_h u_e$ converges to $u_e'$ in $L_2(I)$, that is,
\begin{equation}\label{eq:convinL2weakdiff}
	\|\Delta_h u_e - u_e'\|_{L_2(I)} \to 0,\quad\hbox{as $h\to 0$}.
\end{equation}
Now, observe that, by \eqref{eq:quotdiffderivineq},
\begin{equation}\label{eq:L2ineqquotdiff}
	\|\Delta_h u_e -u_e'\|_{L_2(I)}^2 \leq 2(\|\Delta_h u_e\|_{L_2(I)}^2 + \|u_e'\|_{L_2(e)}^2) \leq 2(C^2+1)\|u_e'\|_{L_2(e)}^2.
\end{equation}
Since $\alpha\geq 2$, it follows by the proof of Theorem \ref{regularity1} and Proposition \ref{prp:hdot1} that $u\in L_2(\Omega, H^1(\Gamma))$, so that
$\pE(\|u_e'\|_{L_2(e)}^2) <\infty$. Therefore, by \eqref{eq:convinL2weakdiff}, \eqref{eq:L2ineqquotdiff}, and the fact that $\pE(\|u_e'\|_{L_2(e)}^2) <\infty$,
we can use the dominated convergence theorem to conclude that
$\pE\bigl(\|\Delta_h u_e - u_e'\|_{L_2(I)}^2\bigr) \to 0$ as $h\to 0$.
Also, by Fubini's theorem, we have that
$$
\int_I \pE(|\Delta_h u_e(t) - u_e'(t)|^2) dt \to 0,\quad\hbox{as $h\to 0$}.
$$
This tells us that $\pE(|\Delta_h u_e(t) - u_e'(t)|^2) \to 0$ in $L_1(I)$ as $h\to 0$. Therefore, there exists a sequence $h_n\to 0$ 
such that $\pE(|\Delta_{h_n} u_e(t) - u_e'(t)|^2)  \to 0$ almost everywhere with respect to the Lebesgue measure in $I$. Hence, for almost every $t\in I$,
$\Delta_{h_n} u_e(t) \to u_e'(t)$, in $L_2(\Omega)$ as $n\to\infty$.
Since $\Delta_{h_n} u_e(t)$ converges strongly in $L_2(\Omega)$, it converges weakly in $L_2(\Omega)$ and thus weakly in $H_u$, so  $\Delta_{h_n} u_e(t) \stackrel{w}{\longrightarrow} u_e'(s)$.	Since we also have that, by Lemma \ref{prp:oddorderderiv}, $\Delta_{h_n} u_e(t)  \stackrel{w}{\longrightarrow} \widetilde{u}_e'(t)$. It follows from the
uniqueness of the weak limit, we have that $u_e'(t) = \widetilde{u}_e'(t)$. Since $I$ contained in the interior of $e$ was arbitrary, it follows that for 
almost every $t$ in $[0,\ell_e]$, $u_e'(t) = \widetilde{u}_e'(t)$. This proves the result for the first derivative. The same argument can be applied to obtain the
result for the higher order derivatives when they exist.
\end{proof}

As a Corollary, we obtain the covariances between the weak derivatives of the solution of \eqref{eq:spde}:

\begin{Corollary}
Let $\alpha\geq 2$ and let $u$ be the solution of \eqref{eq:spde}. Fix any edge $e\in\mathcal{E}$. 
Let $u_e, u_e', \ldots, u_e^{(\lfloor\alpha\rfloor-1)}$ be the directional weak derivatives of $u$. Then, for any
$t,t'\in [0,\ell_e]$, and $j,k \in \{0,\ldots, \lfloor\alpha\rfloor-1\}$, we have 
$$\pE\left(u_e^{(j)}(t) u_e^{(k)}(t')\right) = \frac{\partial^{j+k}}{\partial t^j \partial^k}\pE(u_e(t) u_e(t')).$$
\end{Corollary}
\begin{proof}
First, observe that by Proposition \ref{cor:L2diffweakdiff} the weak derivatives in the Sobolev sense agree with the weak derivatives in the $L_2(\Omega)$ sense
in almost every point of $\Gamma$. Therefore, we can view the derivatives as weak derivatives in the $L_2(\Omega)$ sense. Now, recall the Cameron-Martin space
$\mathcal{H}_u = \dot{H}^\alpha$ (see the proof of Lemma \ref{prp:oddorderderiv}), and note that the function $h(t') = \pE(u_e(t) u_e(t'))$ belongs to $\mathcal{H}_u$. 
Furthermore, it also follows from the proof of Lemma \ref{prp:oddorderderiv} that $h'(t') = \pE(u_e(t)u_e'(t'))$. Hence,
$$
\frac{\partial}{\partial t'}\pE(u_e(t) u_e(t')) = \pE(u_e(t)u_e'(t')).
$$
Now, note that $u_e'(t) \in H_u$, so the function
$g(t') = \pE(u_e'(t) u_e(t'))$ belongs to $\mathcal{H}_u$. We then have that $g'(t') = \pE(u_e'(t) u_e'(t'))$. On the other hand, we have that
$g(t') = \partial \pE(u_e(t) u_e(t')) /\partial t$. Therefore,
$$\frac{\partial^2}{\partial t' \partial t}\pE(u_e(t) u_e(t')) = \pE(u_e'(t)u_e'(t')).$$
The result thus follows by iterating this argument.
\end{proof}

As a final result, we now show that the derivatives satisfy the Kirchhoff vertex conditions, that is, the even-order derivatives are continuous at the vertices and that the
odd-order derivatives satisfy that the sum of the directional derivatives at the vertices is zero.

\begin{Proposition}\label{prp:derivvertexcond}
	Let $\alpha\geq 2$ and let $u$ be a solution of $L^{\nicefrac{\alpha}{2}} (\tau u) = \mathcal{W}$. Fix any $k \in \{0,\ldots, \ceil{\alpha - \nicefrac{1}{2}}  -1\}$. 
	If $k$ is odd, we have, for each $v\in\mathcal{V}$, that $\sum_{e\in\mathcal{E}_v} \partial_{e}^{k} u(v) = 0$. If $k$ is even, we have,
	for each $v\in\mathcal{V}$ and each pair $e,e'\in\mathcal{E}_v$, that $\partial_e^{k} u(v) = \partial_{e'}^{k}u(v)$.
\end{Proposition}
\begin{proof}
	We begin by checking that the first-order directional derivatives satisfy Kirchhoff vertex conditions. To this end, recall the Cameron-Martin space
	$\mathcal{H}_u \cong \dot{H}^\alpha$ from the proof of Lemma \ref{prp:oddorderderiv}. Now, observe that for every vertex $v\in\mathcal{V}$,
	$w = \sum_{e\in \mathcal{E}_v} \partial_e u(v)$ belongs to $H_u$, where $\partial_e u(v) = \pm u_e'(v)$ and the sign is determined by the direction.
	Then, the function $h:\Gamma\to \mathbb{R}$ given by $h(s) = \pE(w u(s))$ belongs to $\mathcal{H}_u$. Since $\alpha\geq 2$, $\dot{H}^\alpha \subset \dot{H}^2$ and,
	by Proposition \ref{prp:hdotk}, $\dot{H}^2 \cong \widetilde{H}^2(\Gamma)\cap C(\Gamma) \cap K(\Gamma)$. Therefore, $h\in K(\Gamma)$. This means that
	$$
	0 = \sum_{e\in\mathcal{E}_v} \partial_{e} h(v) = \pE\Bigl(w \Bigl(\sum_{e\in\mathcal{E}_v} \partial_e u(v)\Bigr)\Bigr) = \pE\Bigl( \Bigl(\sum_{e\in\mathcal{E}_v} \partial_e u(v)\Bigr)^2\Bigr),
	$$
	where the last equality came from our choice of $w\in H_u$. Therefore, $\sum_{e\in\mathcal{E}_v} \partial_e u(v) = 0$. Similarly, take $\alpha \geq 3$
	and for each pair of edges, $e,e'\in\mathcal{E}_v$, let $w = \partial_e^2u(v) - \partial_{e'}^2 u(v) = u_e''(v) - u_{e'}''(v)$. Note that ${w\in H_u}$ and define $g:\Gamma\to\mathbb{R}$ by
	${g(s) = \pE(w u(s))}$. Since $w\in H_u$, we have that $g\in \mathcal{H}_u\cong\dot{H}^\alpha$, which by Proposition \ref{prp:hdotk} shows that 
	$g|_e''(v) = g|_{e'}''(v)$. However, it follows by the proof of Lemma \ref{prp:oddorderderiv}, that $g|_e''(v) = \pE(w u_e''(v))$ and $g|_{e'}''(v) = \pE(wu_{e'}''(v))$.
	Therefore, 
	$$
	0 = g|_e''(v) - g|_{e'}''(v) = \pE(w(u_e''(v) - u_{e'}''(v))) = \pE((u_e''(v) - u_{e'}''(v))^2).
	$$ 
	This gives us $u_e''(v) = u_{e'}''(v)$. 
	The same argument can be applied to obtain the result for higher order derivatives.
\end{proof}

\section{A comparison to isotropic fields}\label{sec:isotropic}

\begin{figure}[t]
\includegraphics[width=0.5\linewidth]{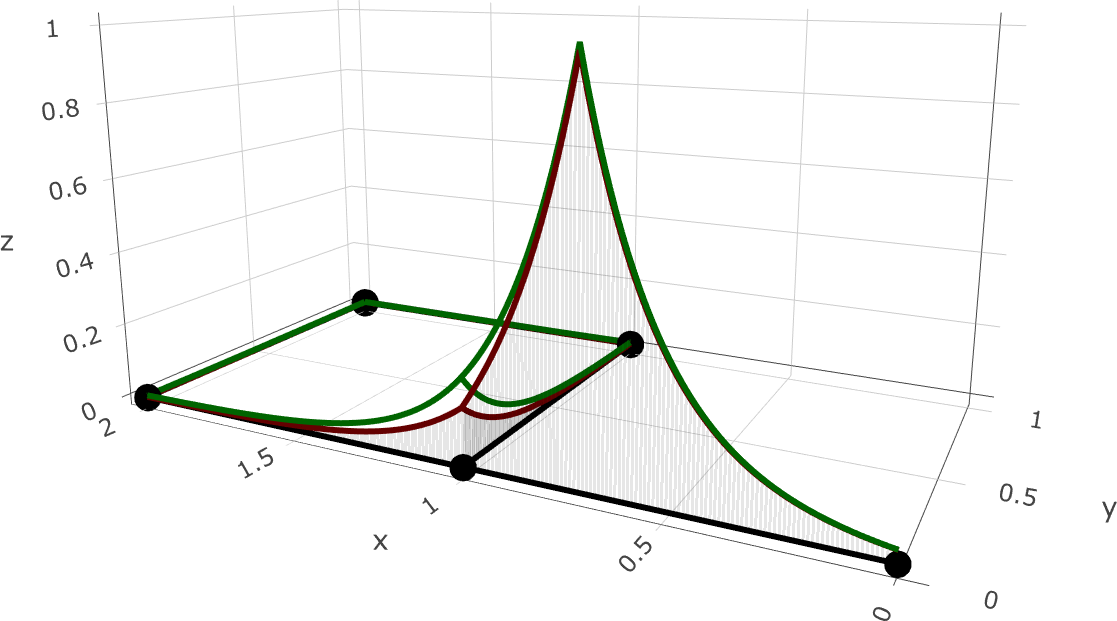}
\caption{The covariance of a Whittle--Mat\'ern field with $\alpha=1$ (red) compared to that of an isotropic Gaussian field with an exponential covariance function in the resistance metric (green). The covariances between the location $0.7$ on the rightmost edge all other locations are shown. }
\label{fig:resistance_compare}
\end{figure}
For graphs with Euclidean edges, one may use the approach by \cite{anderes2020isotropic} to define Mat\'ern-like fields with $\alpha \leq 1$.
In this section, we briefly compare this approach and the Whittle--Mat\'ern fields. 
A comparison between the Whittle--Mat\'ern covariance $\varrho$ with $\alpha = 1$ and the isotropic exponential covariance $\varrho_{exp}$ (based on the resistance metric) is shown in Figure~\ref{fig:resistance_compare}, where the covariance between the field at the location $s_{0}=(1,0.7)$ on the first edge
and all other locations are shown for both models. 

Let $s_{1} = (1,0.3)$ denote a second location on the first edge, let $s_2 = (2,0.1)$ denote a location on the second edge and $s_3 = (3,0.1)$ a location on the third edge. Then $d(s_{1},s_{0}) = d(s_2,s_{0}) = d(s_3,s_{0}),$ and we have that $\varrho(s_0,s_1)  > \varrho(s_0, s_2) = \varrho(s_0,s_3)$ for the Whittle--Mat\'ern model. 
This is a realistic feature for many applications. For example, when modeling traffic intensities in a road network, one might expect that the 
traffic patterns of road segments separated by intersections may be quite different, and one would therefore expect that the correlations 
between locations on the same road segment (edge) a distance $d$ apart are higher compared to correlations between locations a distance $d$ apart on different road segments separated by intersections (vertices of degree higher than two). This is, however, not the case for models with isotropic 
covariance functions since we here have the opposite, $\varrho_{exp}(s_0,s_1)  < \varrho_{exp}(s_0,s_2) = \varrho_{exp}(s_0,s_3)$, 
due to how the resistance metric is defined.  

\begin{figure}[t]
\includegraphics[width=0.3\linewidth]{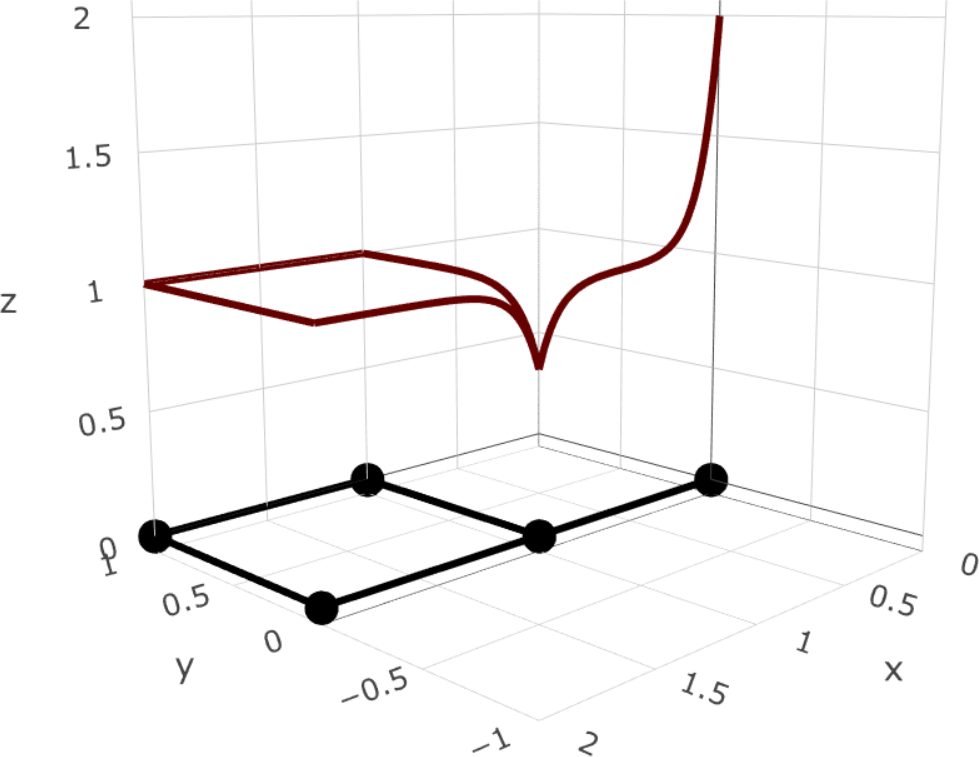}
\includegraphics[width=0.28\linewidth]{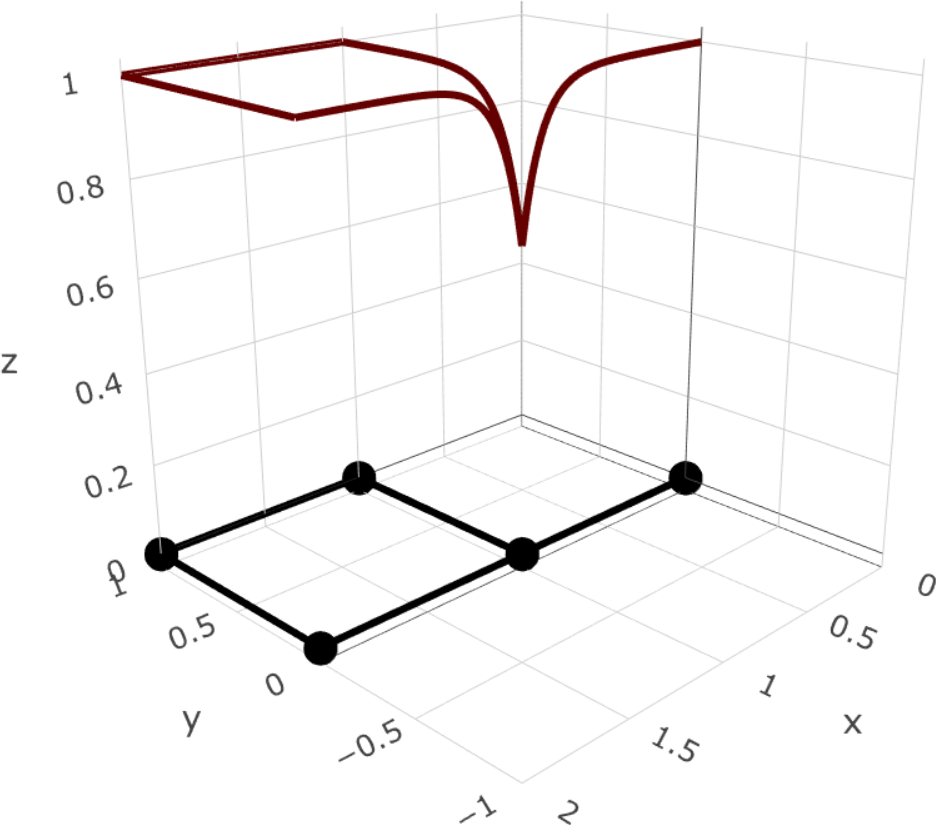}
\caption{Marginal variances of the Whittle--Mat\'ern field with $\alpha=1$ and $\kappa=5$ (left). The right plot shows the corresponding figure with adjusted boundary conditions. }
\label{fig:variances1}
\end{figure}

Thus, compared to the Euclidean setting, it is less clear that isotropy is a desirable property in the setting of metric graphs, where the local structure of the domain itself is varying. 
A consequence of the fact that the Whittle--Mat\'ern fields are not isotropic is that the marginal variances at the vertices will depend on the degrees. In particular, compared to vertices of degree two, the field will have a larger variance at vertices of degree 1 whereas it will have a smaller variance at vertices with higher degrees, see Figure~\ref{fig:variances1}.
The lower variance in vertices of degree higher than two is a natural consequence of the structure of the model. 
Thus, we might not want to remove the non-stationarity for locations close to vertices of degree greater than two. However, for applications, vertices of degree one might be real endpoints in the graph, or they might be endpoints that are induced from observing a part of the entire network. In the latter case, we might want to remove the non-stationarity for such degree-one vertices. This is easy to do by changing the vertex conditions for those vertices, and in particular we can replace the Neumann condition by a Robin boundary condition when $\alpha = 1$. As shown by \cite{Daon2018mitigating}, the Robin boundary conditions $\kappa u + u' = 0$ result in a stationary model on an interval. With this adjustment, we obtain the marginal variances shown in the bottom row of Figure~\ref{fig:variances1}.

\section{Discussion}\label{sec:discussion}
We have introduced the Gaussian Whittle-Mat\'ern random fields on metric graphs and have provided a comprehensive characterization of their regularity properties. We argue that this class of models is a natural choice for applications where Gaussian random fields are needed to model data on metric graphs. 
There are several extensions that we are currently working on for. 
One interesting problem is to characterize the statistical properties of the model, 
and in particular which of the parameters that can be estimated consistently under infill asymptotics.
Another important task is to develop computationally efficient methods for statistical inference based on the model. 
For graphs where one does not know the 
eigenvalues and eigenfunctions explicitly, one option would be to first approximate these and then use Proposition~\ref{prp:KLexp} to simulate from the field, or 
use Proposition~\ref{prp:mercercov} to evaluate the covariance function which in turn can be used to evaluate likelihoods. However, 
a more computationally efficient alternative which will be investigated in future work is to use a finite element method combined with an approximation 
of the fractional operator, similarly to the methods by \cite{BKK2020} and \cite{BK2020rational}, to derive numerical approximations of the random fields. 
This would also facilitate easy extensions to generalized Whittle--Mat\'ern fields, where one could allow the parameters $\kappa$ and $\tau$ to be non-stationary over the graph.  

Another topic that should be investigated in future work is the effect of the chosen vertex conditions. As previously mentioned, the Kirchhoff conditions have many desirable properties, but for certain applications one might have information that warrants the use of other types of vertex conditions. One such example could be that one wants to incorporate information about the angles between the incoming edges in the vertex conditions. Thus, extending the results of this work to models with other vertex conditions is an interesting topic for future work.
Finally, it would also be interesting to consider Type-G extensions of the Gaussian Whittle--Mat\'ern fields similarly to the Type-G random fields in \cite{bw20}. 
A useful property in such a construction is that the process on each edge could be represented as a subordinated Gaussian process.

\begin{appendix}
	
	\section{Real interpolation of Hilbert spaces}\label{app:interpolation}
	
	Let $H_0$ and $H_1$ be two Hilbert spaces, with corresponding inner products $(\cdot,\cdot)_{H_0}$ and $(\cdot,\cdot)_{H_1}$, respectively. Let, also, $H_1\subset H_0$ and for every $u\in H_1$, $\|u\|_{H_1}\geq \| u\|_{H_0}$, where $\|u\|_{H_j}^2 = (u,u)_{H_j}$, $j=0,1$.
	We say that a Hilbert space $H$ is an intermediate space between $H_0$ and $H_1$ if $H_1\subset H \subset H_0$, with continuous embeddings. We say that $H$ is an interpolation space relative to $(H_0,H_1)$ if $H$ is an intermediate space and for every bounded linear operator $T:H_0\to H_0$ such that $T(H_1)\subset H_1$ and the restriction $T|_{H_1}:H_1\to H_1$ is bounded, we also have that $T(H)\subset H$ and $T|_H:H\to H$ is bounded.
	
	We will consider the so-called $K$ method of real interpolation. For further discussions on such methods we refer the reader to \cite{lunardi}, \cite{mclean} or \cite{chandlerwildeetal}.
	Define the $K$-functional for  $t>0$ and $\phi\in H_0$ as
	\begin{equation}\label{eq:kfunctional}
		K(t,\phi; H_0,H_1) = \inf\{(\|\phi_0\|_{H_0}^2 + t^2\|\phi_1\|_{H_1}^2)^{1/2}: \phi_j\in H_j, j=0,1, \phi_0+\phi_1= \phi\},
	\end{equation}
	and for $0<\alpha<1$, define the weighted $L_2$ norm by
	\begin{equation}\label{eq:fracnorm}
		\|f\|_\alpha = \left(\int_0^\infty |t^{-\alpha} f(t)|^2 \frac{dt}{t}\right)^{1/2}.
	\end{equation}
	Then we have the following interpolation space for $0<\alpha<1$,
	\begin{equation}\label{eq:interpolspace}
		(H_0,H_1)_{\alpha} = \{\phi\in H_0: \|K(\cdot, \phi; H_0,H_1)\|_\alpha <\infty\},
	\end{equation}
	which is a Hilbert space with respect to the inner product
	\begin{equation}\label{eq:interpolationinnerproduct}
		(\phi,\psi)_{(H_0,H_1)_\alpha} = \int_0^\infty t^{-2\alpha}K(t,\phi;H_0,H_1)K(t,\psi;H_0,H_1) \frac{dt}{t}.
	\end{equation}

Let, now, $\widetilde{H}_0$ and $\widetilde{H}_1$ be two additional Hilbert spaces, with Hilbertian norms $\|\cdot\|_{\widetilde{H}_0}$ and $\|\cdot\|_{\widetilde{H}_1 }$, respectively, such that $\widetilde{H}_1\subset\widetilde{H}_0$ and for every $u\in \widetilde{H}_1$, we have $\|u\|_{\widetilde{H}_1} \geq \|u\|_{\widetilde{H}_0}$. We say that ${T:H_0\to\widetilde{H}_0}$ is a couple map, and we write $T:(H_0,H_1)\to(\widetilde{H}_0,\widetilde{H}_1)$ if $T$ is a bounded linear operator, $T(H_1)\subset \widetilde{H}_1$, and $T|_{H_1}:H_1\to\widetilde{H}_1$ is also a bounded linear operator.

Given Hilbert spaces $H$ and $\widetilde{H}$, we say that $(H,\widetilde{H})$ is a pair of interpolation spaces relative to $\left((H_0,H_1), (\widetilde{H}_0, \widetilde{H}_1) \right)$ if $H$ and $\widetilde{H}$ are intermediate spaces with respect to $(H_0,H_1)$ and $(\widetilde{H}_0,\widetilde{H}_1)$, respectively, and if, whenever $T:(H_0, H_1)\to (\widetilde{H}_0, \widetilde{H}_1)$ is a couple map, we have that $T(H) \subset \widetilde{H}$ and that $T|_H:H\to \widetilde{H}$ is a bounded linear operator.

The following theorem, which is proved in \cite[Theorem 2.2]{chandlerwildeetal}, tells us that if $H$ and $\widetilde{H}$ are both obtained from the $K$ method, then the pair $(H,\widetilde{H})$ is an interpolation pair:

\begin{Theorem}\label{thm:interpolationpairKmethod}
	Let $0 < \alpha < 1$. Then, $\left((H_0,H_1)_\alpha, (\widetilde{H}_0, \widetilde{H}_1)_\alpha\right)$ is a pair of interpolation spaces with respect to $\left((H_0,H_1), (\widetilde{H}_0, \widetilde{H}_1) \right)$.
\end{Theorem}

Finally we have the following result regarding the equivalence of interpolation spaces, recall that $H_0 \cong H_1$ if $H_0 \hookrightarrow H_1 \hookrightarrow H_0$.
	\begin{Lemma}\label{lemma:interpolationEquivalence}
		Let $H_0,H_1$ be the Hilbert spaces above and let $\bar{H}_1$ be a Hilbert space satisfying $ {\bar{H}_1 \subset H_0}$ and $H_{1} \cong \bar{H}_1$. 
		Then for $0<\alpha<1$ we have that  
		$(H_0,H_1)_{\alpha} 	\cong (H_0,\bar{H}_1)_{\alpha}$.
	\end{Lemma}
	\begin{proof}
		Note that there exist $C_1, C_2>0$ such that for every $u\in H_1$, 
		$C^{-1}_1 \|u\|_{\bar{H}_1} \leq \|u\|_{H_1} \leq C_2 \|u\|_{\bar{H}_1}$.
		Define $\widetilde{C}_1= 1+C_1$ and $\widetilde{C}_2 = 1+C_2$, and observe that
		$K(t,\phi;H_0,  \bar{H}_1)  	\leq  \widetilde{C}_1 K(t,\phi; H_0,H_1)$
		and
		$K(t,\phi; H_0, H_1)  	\leq \widetilde{C}_2 K(t,\phi; H_0, \bar{H}_1).$
		Thus, by combining the above inequalities with \eqref{eq:interpolationinnerproduct},
		we obtain that for any $\phi \in H_0 $,
		$
		\widetilde{C}^{-1}_1  \|u\|_{(H_0,\bar{H}_1)_{\alpha}} \leq \|u\|_{(H_0,H_1)_{\alpha}} \leq \widetilde{C}_2 \|u\|_{(H_0,\bar{H}_1)_{\alpha}}$
	\end{proof}

\section{The fractional Sobolev space $H^\alpha(I)$}\label{app:sob}

In this section we will provide different definitions for the Sobolev space of arbitrary positive order on a bounded interval $I=(a,b)$, and show how they are connected. All the proofs and details of the results in here can be found in \cite{mclean} and \cite{demengel}. The final result on interpolation of Fourier-based fractional Sobolev spaces can be found in \cite{chandlerwildeetal}.

We begin with the definition of the Sobolev spaces of positive integer orders. Let $C_c^\infty(I)$ be the space of infinitely differentiable functions on $I$ with 
compact support.We say that a function $f\in L_2(I)$ has a weak derivative of order $m$ in $L_2(I)$ if there exists a function $v\in L_2(I)$ such that for every $\phi \in C_c^\infty(I)$,
$$\int_I u(t) \frac{d^m \phi(t)}{dt^m} dt = (-1)^m \int_I v(t)\phi(t)dt.$$
In this case we say that $v$ is the $m$th weak derivative of $u$ and we write $v = D^m u$. 

The Sobolev space $H^m(I)$ is defined as
\begin{equation}\label{eq:sobinteger}
	H^m(I) = \{u\in L_2(I): D^j u \hbox{ exists in $L_2(I)$ for every } j=1,\ldots, m\},
\end{equation}
and we have the following characterization of $H^1(I)$:
\begin{equation}\label{eq:sobabscont}
	H^1(I) = \{u\in L_2(I): u \hbox{ is absolutely continuous and $u'\in L_2(I)$}\}.
\end{equation}
Further, in \eqref{eq:sobabscont}, we have that $u' = Du$ a.e., where $u'$ is the a.e. derivative of $u$. 

We will now define the fractional Sobolev-Slobodeckij space of order $0<\alpha<1$. To this end, first we consider the Gagliardo-Slobodeckij semi-norm and the corresponding bilinear form
$$
[u]_{\alpha}^2 = \int_I \int_I \frac{|u(t) - u(s)|^2}{|t-s|^{2\alpha+1}} dtds, \quad 
[u,v]_\alpha = \int_I \int_I \frac{(u(t) - u(s))(v(t)-v(s))}{|t-s|^{2\alpha+1}} dtds.
$$
The Sobolev-Slobodeckij space 
$H^\alpha_S(I) = \{u\in L_2(I): [u]_\alpha <\infty\}$, for ${0<\alpha<1}$, is then
a Hilbert space with respect to the inner product $(u,v)_{H_S^\alpha(I)} = (u,v)_{L_2(I)} + [u,v]_\alpha$.
We have the following Sobolev embedding \cite[Theorem 4.57]{demengel}:

\begin{Theorem}\label{thm:sobolevembeddingR}
	For $1/2 < \alpha \leq 1$, we have that
	$H_S^\alpha(I) \hookrightarrow C^{0,\alpha-1/2}(I)$, where $H_S^1(I):= H^1(I)$.
\end{Theorem}

We also have Fourier-based fractional Sobolev spaces. To define these, let 
$$
\mathcal{S}(\mathbb{R}) = \{u\in C^\infty(\mathbb{R}): \forall a,b\in\mathbb{N}, \sup_{x\in\mathbb{R}} |x^a D^bu(x)| <\infty \}
$$
be the Schwartz space of rapidly decreasing functions, where $C^\infty(\mathbb{R})$ is the space of infinitely differentiable functions in $\mathbb{R}$. We let $\mathcal{S}'(\mathbb{R})$ be the dual of $\mathcal{S}(\mathbb{R})$ and we embed $L_2(\mathbb{R})$ in $\mathcal{S}'(\mathbb{R})$ by the action $\<u,v\> = (u,v)_{L_2(\mathbb{R})},$ where $u\in L_2(\mathbb{R})$ and $v\in \mathcal{S}(\mathbb{R})$. Further, let, for $u\in\mathcal{S}(\mathbb{R})$, $\mathcal{F}(u)$ be the Fourier transform of $u$,
$$\mathcal{F}(u)(x) = \frac{1}{\sqrt{2\pi}} \int_{\mathbb{R}} e^{-ixt} u(t)dt.$$
We can then extend $\mathcal{F}$ to a map from $\mathcal{S}'(\mathbb{R})$ to $\mathcal{S}'(\mathbb{R})$ \cite[see, e.g.,][]{mclean}, and for $\alpha\in\mathbb{R}$, we define  ${H^\alpha_F(\mathbb{R})\subset \mathcal{S}'(\mathbb{R})}$ as
$H^\alpha_F(\mathbb{R}) = \{ u\in \mathcal{S}'(\mathbb{R}): \|u\|_{H^\alpha_F(\mathbb{R})} <\infty\},$
where
$$ \|u\|_{H^\alpha_F(\mathbb{R})} = \left(\int_{\mathbb{R}} (1+|x|^2)^\alpha |\mathcal{F}(u)(x)|^2 dx\right)^{1/2}.$$
Note that the above norm assumes two things: that $\mathcal{F}(u)$ can be identified with a measurable function, and that $(1+|x|^2)^{\alpha/2} |\mathcal{F}(u)(x)|$ belongs to $L_2(\mathbb{R})$.

We can now define the Fourier-based fractional Sobolev spaces on a bounded interval $I$. To this end, let $\mathcal{D}(I)=C_c^\infty(I)$, and $\mathcal{D}'(I)$ be its dual.  Then, we define the Fourier-based Sobolev space of order $\alpha$ as
$H_F^\alpha(I) = \{u\in \mathcal{D}'(I): u = U|_I \hbox{ for some } U \in H_F^s(\mathbb{R})\},$
where $U|_I$ is the restriction of the distribution $U$ to the set $I$ \citep[p. 66]{mclean}. 
The space $H_F^\alpha(I)$ is a Hilbert space with respect to the norm
$\|u\|_{H_F^\alpha(I)} = \inf\{\|U\|_{H_F^\alpha(\mathbb{R})}: U|_I = u\}.$

Recall that two Hilbert spaces $E$ and $F$ are isomorphic, which we denote by $E \cong F$, if $E \hookrightarrow F \hookrightarrow E$. The two definitions of fractional Sobolev spaces are in fact equivalent in this sense \cite[Theorems 3.18 and A.4]{mclean}:

\begin{Theorem}\label{thm:equivfracsob}
	For any bounded interval $I\subset \mathbb{R}$ and any $0\leq \alpha\leq 1$, we have that 
	$H_S^\alpha(I) \cong H_F^\alpha(I),$
	where $H_S^1(I) = H^1(I)$ and $H_S^0(I) = L_2(I)$.
\end{Theorem}

The advantage of the Fourier-based definition of fractional Sobolev spaces is that it is  suitable for interpolation. Indeed, we have by \cite[Lemma 4.2 and Theorem 4.6]{chandlerwildeetal}:

\begin{Theorem}\label{thm:interpolFourierFrac}
	Let $I\subset \mathbb{R}$ be a bounded interval, $\alpha_0\leq \alpha_1$, and $\theta\in (0,1)$. Set $\alpha = \alpha_0 (1-\theta) + \alpha_1\theta$. Then,
	$H_F^\alpha(I) \cong (H_F^{\alpha_0}(I), H_F^{\alpha_1}(I))_{\theta}.$
\end{Theorem}

Finally, for $0<\alpha<1$ we have the following identification of the interpolation-based Sobolev space 
$H^\alpha(I) = (L_2(I), H^1(I))_\alpha.$

\begin{Proposition}\label{prp:identificationFracSobint}
	For $0<\alpha < 1$ and any bounded interval $I$, we have
	$H^\alpha(I) \cong H_S^\alpha(I).$
\end{Proposition}

\begin{proof}
	By Theorem \ref{thm:equivfracsob}, we have that $H_F^0(I) \cong L_2(I)$ and $H_F^1(I)\cong H^1(I)$. Thus, it follows from the same arguments as in the proof of 
	Corollary \ref{cor:identificationHalpha} 
	that
	$$H^\alpha(I) = (L_2(I), H^1(I))_\alpha \cong (H_F^0(I), H_F^1(I))_\alpha.$$
	Now, by this equality and Theorems \ref{thm:equivfracsob}--\ref{thm:interpolFourierFrac}, we have that
	$H^\alpha(I) \cong H_F^\alpha(I) \cong H_S^\alpha(I).$
\end{proof}

As an immediate consequence from Proposition \ref{prp:identificationFracSobint} and Theorem \ref{thm:sobolevembeddingR}
we obtain a Sobolev embedding for the interpolation-based fractional Sobolev space:

\begin{Corollary}\label{cor:SobembedInterpFracSob}
	For $1/2 < \alpha \leq 1$, we have that
	$H^\alpha(I) \hookrightarrow C^{0,\alpha-1/2}(I).$
\end{Corollary}

\end{appendix}

\bibliographystyle{imsart-number}
\bibliography{../../Bib/unified_graph_bib}

\end{document}